\documentclass[reqno]{amsart}
\usepackage{amsmath,amsbsy,amsfonts}
\usepackage{color}

\date{}
\def\lb{\lambda}
\def\var{\varepsilon}
\def\pil{\left<}
\def\pir{\right>}
\def\dis{\displaystyle}
\def\nd{\noindent}
\def\thend{\rule{3mm}{3mm}}
\def\Re{\mathbb{R}}
\newtheorem{theorem}{Theorem}[section]

\newtheorem{prop}{Proposition}[section]
\newtheorem{lem}{Lemma}[section]
\newtheorem{rmk}{Remark}[section]

\newcommand{\theend}{\hfill $\Box$}
\newcommand{\w}{W_0^{1,\Phi}(\Omega)}
\renewcommand{\theequation}{\arabic{section}.\arabic{equation}}

\newcommand{\cqd}{\hspace{10pt}\fbox{}}
\newcommand{\cqdf}{\hspace{10pt}\rule{5pt}{5pt}}
\newcommand{\Int}{\displaystyle\int_{\Omega}}
\newcommand{\eps}{\epsilon}
\newcommand{\ds}{\displaystyle}
\newcommand{\Fr}{\displaystyle\frac}
\newcommand{\Proof}{{\hspace{-.8cm}\bf Proof: }}
\newcommand{\Z}{\tilde{z}}
\newcommand{\U}{\tilde{u}}
\newcommand{\V}{\tilde{v}}
\newcommand{\Q}{\tilde{Q}}
\newcommand{\hu}{\hat{u}}
\newcommand{\hv}{\hat{v}}
\newcommand{\E}{\bar{\eps}}
\newcommand{\2}{2^*}
\newcommand{\on}{o_n(1)}
\newcommand{\el}{\ell^*}
\newcommand{\s}{\ds\sum_{i=1}^{r}}
\newcommand{\Rn}{\mathbb{R}^N}
\newcommand{\R}{\mathbb{R}}
\newcommand{\M}{\cal{M}}
\newcommand{\N}{\mathcal{N}_{\lambda,\mu}}
\begin{document}
\title[Ground and bound state solutions for quasilinear elliptic systems including singular nonlinearities] {Ground and bound state solutions for quasilinear elliptic systems including singular nonlinearities and indefinite potentials}
\vspace{1cm}

\author{M. L. M. Carvalho}
\address{M. L. M. Carvalho \newline Universidade Federal de Goias, IME, Goi\^ania-GO, Brazil }
\email{\tt marcos$\_$leandro$\_$carvalho@ufg.br}

\author{Edcarlos D. da Silva}
\address{Edcarlos D da Silva \newline  Universidade Federal de Goias, IME, Goi\^ania-GO, Brazil}
\email{\tt edcarlos@ufg.br}

\author{C. A. Santos}
\address{C. A. Santos \newline  Universidade de  Brasilia, Brasilia-DF, Brazil}
\email{\tt c.a.p.santos@mat.unb.br}

\author{C. Goulart}
\address{C. Goulart \newline Universidade Federal de Jata\'\i, Jata\'\i-GO, Brazil }
\email{\tt claudiney@ufg.br}

\subjclass{35J20, 35J25, 35J60, 35J92, 58E05} \keywords{Quasilinear elliptic systems, Nonhomogeneous operators, Nehari method, Indefinite nonlinearities, Nonsingular nonlinearities, Singular nonlinearities}
\thanks{The authors was partially supported by Fapeg/CNpq grants 03/2015-PPP}
\begin{abstract}
It is established existence of bound and ground state solutions for quasilinear elliptic systems driven by $(\Phi_{1}, \Phi_{2})$-Laplacian operator. The main feature here is to consider quasilinear elliptic systems involving both  nonsingular nonlinearities combined with indefinite potentials and singular cases perturbed by  
 superlinear and subcritical couple terms. These prevent us to use arguments based on Ambrosetti-Rabinowitz condition and variational methods for differentiable functionals. By exploring the Nehari method and doing a fine analysis on the fibering map associated, we get estimates that allow us unify the arguments to show multiplicity of semi-trivial solutions in both cases. 
\end{abstract}

\maketitle

\section{Introduction}

In this work we consider  the class of quasilinear elliptic system  driven by the $(\Phi_1,\Phi_2)$-Laplacian operator in the  form
\begin{equation}\label{eq1}
\left\{\begin{array}{rcl}
-\Delta_{\Phi_1} u  & = & \lambda a(x)|u|^{q-2}u  + \frac{\alpha}{\alpha+\beta}b(x)|u|^{\alpha-2}u|v|^{\beta} \, \mbox{ in }\,    \Omega, \\[1,3ex]
-\Delta_{\Phi_2} v  & = & \mu c(x)|v|^{q-2}v  + \frac{\beta}{\alpha+\beta}b(x)|u|^{\alpha}|v|^{\beta-2}v \, \mbox{ in }\,   \Omega, \\[1,3ex]
u&=&v \,\, = \,\,0 \,\,\,\, \mbox{on}\,\,\,\, \partial\Omega ,
\end{array} \right.
\end{equation}
where $\Omega \subset \mathbb{R}^{N}$ is a smooth bounded  domain with $N \geq 2$ and $
\Delta_{\Phi_i} u =  \mbox{div} ( \phi_i(|\nabla u|) \nabla u)
$
with 
\begin{equation*}
\Phi_i(t) := \int_{0}^{\vert t \vert} s \phi_i(s) ds, t \in \mathbb{R}, i = 1,2
\end{equation*}
for some $C^{2}$-function $\phi_i: (0,\infty)\rightarrow (0,\infty)$ whose assumptions will be established later.  We also consider  $a,b,c: \Omega \rightarrow \mathbb{R}$ are continuous potentials in $L^{\infty}(\Omega)$; $\alpha,\beta > 1$ and $ 0 < q < \ell_{i} \leq m_{i} < \min\{\ell_{i}^{*}\}$, where $\ell_{i}, m_{i} \in (1, N)$ and $\ell_{i}^{*} = \ell_{i} N /(N - \ell_i)$, for $ i = 1,2$, that lead us to a class of singular or nonsingular systems with superlinear-subcritical couple terms. Our principal goal is showing existence  of ground state solution (solution that has minimum energy among any nontrivial solutions) and bound state solution (solution with  finite energy). 

Non-homogeneous differential operators  have been widely considered in the literature  in the context of scalar problems (see  \cite{rad2,rad3,appl1,Hainfan2014,Fuk_1,Fuk_2,Garcia,rad1} for further details); however, there are few works dealing with systems of the type (\ref{eq1}). In  \cite{HM} it was considered  an eigenvalue problem for the $(\Phi_1, \Phi_2)$-Laplacian, while \cite{QiLin} addressed to  $(p,q)-$Laplacian framework. Unlike to this case, there are several works in the literature dealing with homogeneous operator for problems of the type (\ref{eq1}) for the nonsingular case. Even to this case,  there are few works in the context of singular ones.

As we have already said, for the nonsingular cases there are a variety of works treating Problem (\ref{eq1}) with different kinds of potentials and nonlinearities. In  \cite{Hsu} and  \cite{QiLin} were considered definite potentials, while in \cite{Hainfan2014} and \cite{wu-sytem} were studied in the setting of indefinite potentials. This case in more challenging due the lake of  Ambrosetti-Rabinowitz condition. For more details envolving these feature on potentials under    subcritical, critical and supercritical behavior of the couple term,
 see  \cite{Alana,Alana2,Hainfan2014,Garcia,Hsu2012,HM,depaiva,Ramos}  and references therein.

About singular elliptic systems, the are few results dealing problem of the type (\ref{eq1}). By using non-variational methods, the works \cite{alves,JT,elmanouni,Carvalho2,hai} and references therein showed existence of solutions, but they did not get multiplicity results. The principal difficulty in approaching problems like  (\ref{eq1}) with variational methods meets in the fact the energy functional may be of infinite energy in the whole space. Depending on how strong the singularity is, we can have finite energy for the functional either in some parts of the base space or in the whole one, but it will never  be Gateaux differentiable in the whole space. By constraining our energy functional on a set of the type Nehari and exploring ideas found in  \cite{yijing, yijing2011,yijing2008}, we are able to show that the energy functional is Gateaux differentiable at the minimum points of this functional constrained some disjoint subsets of this Nehari set.

Besides the mathematical interest about problems with $(\Phi_1, \Phi_2)$-Laplacian operator, see for instance \cite{ Carvalho2,HM}, there is a wide amount of real-world problems   modeled by operators of the type $(\Phi_1, \Phi_2)$-Laplacian, for instance, in the fields of non-Newtonian fluids, image processing, plasma physics, among others, see \cite{Fuk_1,Fuk_2}. Just to highlight some of them, we mention some situations in what $\Phi = \Phi_{1} = \Phi_{2}$ was considered:
\begin{itemize}
	\item[(i)] nonlinear elasticity: $\Phi(t)=(1+t^{2})^{\gamma}-1$, $1<\gamma<N/(N-2)$;
	\item[(ii)] plasticity: $\Phi(t)=t^{\alpha}(\log(1+t))^{\beta}$, $\alpha\geq1$, $\beta>0$;
	\item[(iii)] non-Newtonian fluid: $\Phi(t)=\frac{1}{p}|t|^{p}$, for $p>1$;
	\item[(iv)] plasma physics: $\Phi(t)=\frac{1}{p}|t|^{p}+\frac{1}{q}|t|^{q}$, where $1<p<q<N$ with $q\in(p,p^{*})$;
	\item[(v)] generalized Newtonian fluids: $\Phi(t)=\int_{0}^{t}s^{1-\alpha}[\sinh^{-1}(s)]^{\beta}\;\mathrm{d}s$, $0\leq\alpha\leq 1$, $\beta>0$.
\end{itemize}

To introduce the setting spaces to approach Problem $(\ref{eq1})$ by Variational methods, let us  consider 
$C^{2}$-functions $\phi_i: (0,\infty)\rightarrow (0,\infty)$ satisfying:
\begin{itemize}
  \item[$(\phi_1)$] $\displaystyle \lim_{t \rightarrow 0} t \phi_i(t)= 0, \displaystyle \lim_{t \rightarrow \infty} t \phi_i(t)= \infty$,
  \item[$(\phi_2)$] $t\mapsto t\phi_i(t)$ is strictly increasing,
  \item[($\phi_3$)] $ -1<\ell_i-2:=\ds\inf_{t>0}\Fr{(t\phi_i(t))^\prime t}{(t\phi_i(t))^\prime}\leq \ds\sup_{t>0}\Fr{(t\phi_i(t))^{\prime\prime}t}{(t\phi_i(t))^\prime}=:m_i-2< N-2$
\end{itemize}
for $ i = 1,2$ and note that all the functions listed in the examples $(i)$-$(v)$ satisfy the assumptions $(\phi_1)-(\phi_3)$, where $\phi$ is such that $\Phi'(t)=\Phi'_i(t) =\phi_i(t)t=\phi(t)t,~t\geq 0$. 

Under hypotheses $(\phi_1)-(\phi_3)$ and due to the nature of the $
\Delta_{\Phi_i}  $-operator, it is natural to work on reflexives and Banach spaces called Orlicz and Orlicz-Sobolev spaces, which will be denoted by $L_{\Phi_i}(\Omega)$ and $W^{1, \Phi_i}_{0}(\Omega)$, respectively. These hypotheses may introduce Orlicz spaces that are not equivalent to any Lebesgue spaces. One well known example is the N-function
$$\Phi_i(t)=|t|^{\ell_i}ln(|t|+1),~t\in \mathbb{R}$$
that satisfies the assumptions $(\phi_1)-(\phi_3)$  with $m_i=\ell_i +1$ and $\ell_i>1, i = 1,2$, but $L_{\Phi_i}(\Omega)$ is not equivalent to any Lebesgue space $L^s(\Omega)$ for any $s\geq1$. In particular, the approach of  the quasilinear elliptic System (\ref{eq1}) on these spaces allows us to deal with  operators of the type $-\Delta_{p_i}u -\Delta_{q_i}u$ with $p_i,q_i>1$, as well. See the Appendix for  additional details on Orlicz and Orlicz-Sobolev framework. 

Another important consequence of $(\phi_1)-(\phi_3)$ is the inequality
\begin{equation}
\label{22}
\ell_i\leq\Fr{\phi_i(t)t^2}{\Phi_i(t)}\leq m_i~\mbox{and }
\ell_i-2\leq \frac{\phi^\prime(t)t}{\phi(t)}\leq m_i-2,
\end{equation}
which, together with the assumption
\begin{flushleft}
$(H)~~{0}<q<\Fr{(\alpha+\beta-1)\min\{\ell_i\}-\max\{m_i(m_i-1)\}}{\alpha+\beta-\min\{\ell_i\}} \leq \ell_i\leq m_i<\alpha+\beta< \min\{\ell_i^*\}, i = 1,2,$
\end{flushleft}
lead us to infer that
\begin{equation*}
0<q<\Fr{\ell_i(\alpha+\beta-m_i)}{\alpha+\beta-\ell_i}\leq \ell_i\leq m_i<\alpha+\beta< \min\{\ell_i^*\}~\mbox{for } i = 1,2,
\end{equation*}
holds true.

To define an energy functional, let us denote by  $W:= W^{1, \Phi_1}_{0}(\Omega)\times W^{1, \Phi_2}_{0}(\Omega)$ that is a reflexive Banach space endowed with the  norm
$$||z|| =||u||_{W^{1,\Phi_1}_0}+||v||_{W^{1,\Phi_2}_0}:=||u||+||v||,$$
where $z=(u,v) \in W$.  Under the above assumptions, it is standard to show that the energy function $J:W 
\to \mathbb{R}$ associated to the system (\ref{eq1}), defined by  
\begin{equation*}
J(z)=\Int A(z)-\Fr{1}{q}  K(z)-\Fr{1}{\alpha+\beta} Q(z),~~z=(u,v)\in W,
\end{equation*}
is well-defined, but it may not be Gateaux differentiable in the whole space. When $q>1$, it is a $C^1$-functional  and its derivative is given by 
\begin{equation*}\begin{array}{rcl}
 \langle J^\prime(z), \varphi \rangle = \Int  A^\prime(z)\varphi
 - \dfrac{1}{q}  K^\prime(z)\varphi
 - \Fr{1}{\alpha+\beta}Q^\prime(z)\varphi\,\, \mbox{for any} \,\, z, \varphi \in W,
 \end{array}
\end{equation*}
where 
$$A(z)=\Phi_1(|\nabla u|)+\Phi_2(|\nabla v|),~~K(z)=\lambda a(x)|u|^q+\mu c(x)|v|^q,~~Q(z)= b(x)|u|^\alpha|v|^\beta.$$
and their derivatives are given by
$$A^\prime(z)\varphi=\phi_1(|\nabla u|)\nabla u\nabla \varphi_1+\phi_2(|\nabla v|)\nabla v\nabla \varphi_2,~~K^\prime(z)\varphi=\lambda q a(x)|u|^{q-2}u\varphi_1+\mu q c(x)|v|^{q-2}v\varphi_2,$$
and
$$ Q^\prime(z)\varphi= b(x)( \alpha|u|^{\alpha-2}u|v|^\beta \varphi_1+\beta|v|^{\beta-2}v|u|^\alpha \varphi_2).$$

In both cases ($0<q<1$ and $q>1$), let us show that finding weak solutions to System \eqref{eq1} is equivalent to get critical points to the functional $J$, that is, a weak solution $z=(u,v) \in W$ to  the quasilinear elliptic system \eqref{eq1} means that
\begin{equation*}\begin{array}{rcl}
 \Int  A^\prime(z)\varphi
= \Int\dfrac{1}{q}  K^\prime(z)\varphi
+\Fr{1}{\alpha+\beta}  Q^\prime(z)\varphi \mbox{ for all } \varphi=(\varphi_1,\varphi_2)\ \in W.\end{array}
\end{equation*}
This implies that $0=(0,0)$ is a solution of the System \eqref{eq1}, called trivial, while solutions of the type $z_1=(u,0)$ or $z_2=(0,v)$ are named as semitrivial solutions. Finely, $z=(u,v)\geq 0 \ (>0)$ is non-negative (positive) solution, whose meaning is $u,v\geq 0\  (u,v>0)$. 


To state our principal results, let us assume the  assumptions:
\begin{itemize}
\item[$(A)$] $b$ is a continuous function satisfying $||b||_\infty = 1$ and $b^{+} \neq 0$,
\smallskip
\item[$(B)$] $a,c$ are also continuous functions  that satisfy $||a||_\infty=||c||_\infty=1$, $a^{+} \neq 0$ and $c^{+} \neq 0$.
\end{itemize}

We have.
\begin{theorem}[Nonsingular Case]\label{teorema1} Assume that $(\phi_{1}) - (\phi_{3})$, $(A)$, $(B)$ and $(H)$ hold. If $q >1$, then there exists a $\lambda_{\star}> 0$ such that  System \eqref{eq1} admits at least two nonnegative solutions for each $\lambda,\mu\geq 0$ given satisfying $0<\lambda+\mu \leq   \lambda_{\star}$. One of them is a ground state solution $\bar{z}_{\lambda, \mu}$ and the other one a bound state solution  $\tilde{z}_{\lambda, \mu}$. Furthermore, we obtain   that
$\bar{z}_{\lambda, \mu},\tilde{z}_{\lambda, \mu}\in W \setminus\{0,z_1,z_2\}$,  $J(\bar{z}_{\lambda, \mu}) < 0<J(\bar{z}_{\lambda, \mu}) $ and $\displaystyle \lim_{\lambda, \mu \rightarrow 0^{+}} \|\bar{z}_{\lambda, \mu}\| = 0$.
\end{theorem}

\noindent For the singular case, let us consider the assumption.
\begin{itemize}
\item[$(C)$] $a,c$ and $b$ are nonnegative continuous functions satisfying $||a||_\infty=||b||_\infty=||c||_\infty=1$.
\end{itemize}

\begin{theorem}[Singular Case]\label{teorema2} Assume that $(\phi_{1}) - (\phi_{3})$, $(C)$ and $(H)$ hold. If $0 < q < 1$, then there exists a $\lambda^{\star}> 0$ such that System \eqref{eq1} admits at least two positive solutions for each $ \lambda,\mu\geq 0$ given satisfying $0<\lambda+\mu<\lambda^{\star}$. One of them is a ground state $\bar{z}_{\lambda, \mu}$ and the other one a bound state $\tilde{z}_{\lambda, \mu}$.  Moreover, $c d\leq \bar{z}_{\lambda, \mu}, \tilde{z}_{\lambda, \mu} \in W \setminus\{0,z_1,z_2\}$, $J(\bar{z}_{\lambda, \mu}) < 0<J(\tilde{z}_{\lambda, \mu}) $ and $\displaystyle \lim_{\lambda,\mu\rightarrow 0^{+}} \|\bar{z}_{\lambda, \mu}\| = 0$, for some real constant $c>0$, where $d(x):=\mbox{dist}(x,\partial \Omega), x \in \Omega$ is the distance function to the boundary of $\Omega$.
\end{theorem}

As consequence of our results, we obtain existence of non-negative ground and bound state solutions (for $q >1$) and positive ground and bound state solutions (for $0<q<1$) for several quasilinear elliptic systems. Just to highlight this, let us consider the two below classes. As a first example, we have the   system with the $(p_{1}, p_{2})$-Laplacian operator
\begin{equation*}
\left\{\begin{array}{rcl}
- \Delta_{p_{1}} u  & = & \lambda a(x)|u|^{q-2}u  + \frac{\alpha}{\alpha+\beta}b(x)|u|^{\alpha-2}u|v|^{\beta} \, \mbox{ in }\,    \Omega, \\[1ex]
-\Delta_{p_{2}} v  & = & \mu c(x)|v|^{q-2}v  + \frac{\beta}{\alpha+\beta}b(x)|u|^{\alpha}|v|^{\beta-2}v \, \mbox{ in }\,   \Omega, \\[1ex]
u&=&v  = 0 \,\, \mbox{on}\,\, \partial\Omega 
\end{array} \right.
\end{equation*}
with $\Omega \subset \mathbb{R}^{N}$ being a bounded smooth domain, $N \geq 2$, $1 < p_{1} \leq p_{2} < N$ and either $0<q <1$ or $1 \leq q < p_{1} \leq p_{2} < \alpha + \beta < \min \{p_{i}^{*}\}$, where $p_{i}^{*} = p_{i}N/ (N - p_{i}), i = 1,2$. 

Another example that is reached by our theorems  is the system with $p_{i}\verb"&"r_{i}$-Laplacian operator. More specifically, 
\begin{equation*}
\left\{\begin{array}{rcl}
- \Delta_{p_{1}} u  - \Delta_{r_{1}} u   & = & \lambda a(x)|u|^{q-2}u  + \frac{\alpha}{\alpha+\beta}b(x)|u|^{\alpha-2}u|v|^{\beta} \, \mbox{ in }\,    \Omega, \\[1ex]
-\Delta_{p_{2}} v  - \Delta_{r_{2}}v & = & \mu c(x)|v|^{q-2}v  + \frac{\beta}{\alpha+\beta}b(x)|u|^{\alpha}|v|^{\beta-2}v \, \mbox{ in }\,   \Omega, \\[1ex]
u&=&v  = 0 \,\, \mbox{on}\,\, \partial\Omega 
\end{array} \right.
\end{equation*}
with  $1 < p_{i} \leq r_{i} < N$,  $p_{i} \leq r_{i} < \alpha + \beta < \min\{p_{i}^{*}\}$ and either $0<q <1$ or $1<q < p_i$. The main interest point here is to consider the case in what $p_{1},r_{1}, p_{2}, r_{2}$ are different. 

 Below, let us highlight some of the main contributions of this paper to the current literature:
\begin{enumerate}
\item[$(i)$] we deal with all difficulties that arise from the nature of the $(\Phi_{1},\Phi_{2})$-Laplace operator and so this enable us to extend the former and classical results to a wide class of  non-homogeneous operators as well, even for indefinite potentials $a,b,c: \Omega \rightarrow \mathbb{R}$ in the nonsingular case,
\item[$ii)$] Theorem \ref{teorema1} gathers different classes of problems to exhibit multiplicity of nonnegative solutions by unifying various  approaching in early literature,
\item[$iii)$] we fit some approaches in the context of homogeneous operators to that one non-homogeneous. One sensible point is to do this on the  strategy of Yijing \cite{yijing},
\item[$iv)$] Theorem \ref{teorema2} guarantee not only a multiplicity result of positive solutions to the singular problem \eqref{eq1} but principally existence of a positive ground state and a bound state solutions. It is new even for Laplacian operator. 
\end{enumerate}

To ease our future references, let us set some notations:	
\begin{itemize}
	\item $C$, $\tilde{C}$, $C_{1}$, $C_{2}$,... denote positive constants (possibly different).
	\item $o_{n}(1)$ denotes a sequence that converges to $0$ as $n\rightarrow\infty$;
	\item $\displaystyle\lim_{t\rightarrow 0^+}f(t)$ denotes the right-hand limit as $t \rightarrow 0$  for any function $f:\Omega \rightarrow \mathbb{R}$;
    \item the norms in $L^{s}(\Omega)$, for $1\leq s<\infty$, and $L^{\infty}(\Omega)$ will be denoted by $\|\cdot\|_{s}$ and $\|\cdot\|_{\infty}$,  respectively.
    \item it will be considered on $L^{s}(\Omega)\times L^{s}(\Omega)$ the norm $\|(u,v)\|_{s}=\left(\|u\|^{s}_{s}+\|v\|^{q}_{s}\right)^{1/s}$,
 \item for $z,\varphi \in W$, with $z=(u,v)$ and $\varphi=(\varphi_1,\varphi_2)$, we define:
    \begin{itemize}
    \item $A(z)=\Phi_1(|\nabla u|)+\Phi_2(|\nabla v|)$,
	\item $ A^\prime(z)\varphi=\phi_1(|\nabla u|)\nabla u\nabla \varphi_1+\phi_2(|\nabla v|)\nabla v\nabla \varphi_2$,
	\item $B(z)= A^\prime(z)z=\phi_1(|\nabla u|)|\nabla u|^2+\phi_2(|\nabla v|)|\nabla v|^2$,
	\item $C(z)\varphi=\phi_1^\prime(|\nabla u|)|\nabla u|\nabla u\nabla \varphi_1+\phi_2^\prime(|\nabla v|)|\nabla v|\nabla v\nabla \varphi_2$
	\item $D(z)=C(z)z=\phi_1^\prime(|\nabla u|)|\nabla u|^3+\phi_2^\prime(|\nabla v|)|\nabla v|^3$,
	\item $K(z)=\lambda a(x)|u|^q+\mu c(x)|v|^q$,
	\item $K^\prime(z)\varphi=\lambda q a(x)|u|^{q-2}u\varphi_1+\mu q c(x)|v|^{q-2}v\varphi_2$,
    \item $Q(z)= b(x)|u|^\alpha|v|^\beta \mbox{~and } Q^\prime(z)\varphi= b(x)( \alpha|u|^{\alpha-2}u|v|^\beta \varphi_1+\beta|v|^{\beta-2}v|u|^\alpha \varphi_2)$,
	\item $Q^\prime(z)\varphi= b(x)( \alpha|u|^{\alpha-2}u|v|^\beta \varphi_1+\beta|v|^{\beta-2}v|u|^\alpha \varphi_2)$,
	\item $Q^\prime(z)z=(\alpha+\beta)Q(z)$,
	\item $H(z)= A(z)-\Fr{1}{\alpha+\beta}B(z)$,
    \end{itemize}
 \item   the $\bar{S_i},S_i,S_i^*$ denote the best Sobolev constants for the embedding
$W^{1,\Phi_i}(\Omega)\hookrightarrow L^q(\Omega)$, whenever $q> 1$, $W^{1,\Phi_i}(\Omega)\hookrightarrow L^{\alpha+\beta}(\Omega)$ and $W^{1,\Phi_i}(\Omega)\hookrightarrow L^{\ell_i^*}(\Omega)$, respectively,
\item $\bar{S}=\max\{\bar{S_i}\},~{S}=\max\{{S_i}\}$ and ${S_*}=\max\{{S_i^*}\}$,
    \item the functions $a^{+} = max(a, 0)$ and $a^{-} = min(a,0)$ stand for the positive and negative parts of each $a \in L^{\infty}(\Omega)$ given.
\end{itemize}

The paper is organized as follows: in the forthcoming Section \ref{sec-nehari}, we consider the Nehari method for the quasilinear elliptic System \eqref{eq1}, while Section \ref{fib-map} is devoted to study the fibering map linked to the Nehari manifold. In Section \ref{main-res} we present the proof of our main results. In the last section (in the appendix) we gather some basic topics on Orlicz and Orlicz-Sobolev spaces to ease the reading of the reader.


\section{The Nehari manifold}\label{sec-nehari}

In this section we shall prove some properties for the Nehari method in the context of systems considering the Orlicz-Sobolev setting. The main feature here
is to give information on the critical points for fibering map associated to the energy functional $J$. For an overview on the Nehari
method we infer the reader to the interesting works \cite{Brown2,Brown1}. The first feature here is to consider the fibering map $\gamma_z: [0,+\infty)\to \mathbb{R}$ given by
$$\gamma_z(t):=J(tz)=\Int A(tz)-\Fr{ t^{q}}{q}K(z)-\Fr{t^{\alpha+\beta}}{\alpha+\beta}Q(z), z\in W\setminus \{(0,0)\}.$$
The fibering map has been considered together with the Nehari manifold in order to ensure the existence of critical points for $J$. For concave-convex nonlinearities is important a great knowledge around the geometry of $\gamma_{z}$. Here we infer the reader to important works on the Nehari method \cite{Brown2,Brown1,Drabek,wu,wu3}.

First of all, we shall consider the nonsingular case, that is, we consider $q > 1$. Latter on, we shall discuss the singular case showing some useful tools in order to give a description on the fibering maps. 
Now, for the nonsingular case we point out that $\gamma_{z}: (0, \infty) \rightarrow \mathbb{R}$ is in $C^{1}$ class thanks to hypotheses $(\phi_{1}) - (\phi_{2})$. More specifically, we obtain
\begin{equation*}
\begin{array}{rcl}
\gamma^\prime_z(t)&=& \Int t^{-1}B(tz)-{ t^{q-1}}K(z)-{t^{\alpha+\beta-1}}Q(z),\ \ t>0.
\end{array}
\end{equation*}
The Nehari manifold associated to the energy functional $J$ is defined by
\begin{equation}\label{nehari}
\begin{array}{rcl}
\mathcal{N}_{\lambda,\mu}&=& \left\{z\in W\setminus \{0\}: \gamma^\prime_z(1)=0 \right\} = \left\{z\in W\setminus \{0\}:\Int B(z)=\Int K(z) + Q(z)\right\}.
\end{array}
\end{equation}
Notice that when $z$ is a nontrivial weak solution of System $\eqref{eq1}$ we obtain that $z \in \N$. Moreover, using \eqref{nehari},  for any $z\in \mathcal{N_{\lambda,\mu}}$ we obtain

\begin{equation}\label{eq2}
J(z)=\Int A(z) - \Fr{1}{\alpha+\beta} B(z)+\left(\Fr{1}{\alpha +\beta}-\Fr{1}{q}\right)K(z),
\end{equation}
or equivalently
\begin{equation*}\label{eq3}
J(z)= \Int A(z) - \Fr{1}{q}B(z)+\left(\Fr{1}{q}-\Fr{1}{\alpha +\beta}\right)Q(z).
\end{equation*}
It is easy to see that $tz\in\mathcal{N}_{\lambda,\mu}$ if and only if $\gamma^\prime_z(t)=0, z \neq 0$.  Therefore, $z \in\mathcal{N}_{\lambda,\mu}$ if and  only if $\gamma^\prime_z(1)=0$. In other words, it is sufficient to find stationary points of the fibering map in order to get critical points for $J$ on $\mathcal{N}_{\lambda,\mu}$.
Note that $J$ is not in $C^2$ class in general due the fact that the operator $(\Phi_{1}, \Phi_{2})$-Laplacian can be singular at the origin. Hence the second derivative is not well defined for any direction $(h_1, h_{2}) \in W$. On the other hand, using hypothesis $(\phi_{3})$, for the nonsingular case we deduce that $t \mapsto \gamma_{z}(t)$ is in $C^{2}$ class for any $t > 0$ with second derivative given by
\begin{eqnarray*}
	\gamma^{\prime\prime}_z(t)= &\Int t^{-1}D(tz)+t^{-2} B(tz)
	-(q-1)t^{q-2}K(z)
	+ (\alpha+\beta-1)t^{\alpha+\beta-2} Q(z).
\end{eqnarray*}
In general, applying hypothesis $(\phi_{3})$, we mention that the map $z \mapsto J^{\prime \prime}(z)(z,z)$ is well defined for each $z \in W$ which provides us a continuous function. In particular, we know that $\gamma^{\prime \prime}_{z}(1) = J^{\prime \prime}(z)(z,z)$ for any $z \in \mathcal{N}_{\lambda,\mu}$.

As was pointed by Brown et al \cite{Brown2,Brown1} it is natural to split $\N$ into three sets as follows:
$$\N^+:=\{z\in \N:\gamma^{\prime\prime}_z(1)>0\};$$
$$\N^-:=\{z\in \N:\gamma^{\prime\prime}_z(1)<0\};$$
$$\N^0:=\{z\in \N:\gamma^{\prime\prime}_z(1)=0\}.$$
Here we mention that $\N^+,~\N^-,~\N^0$ corresponds to critical points of minimum, maximum and inflexions points for the fibering map $\gamma_z$, respectively. On this subject we refer the interesting reader also to Tarantello \cite{tarantello}.

\begin{rmk}\label{gamma''}
	It is no hard to verify that
	\begin{equation}\label{eq5}
	\begin{array}{rcl}
	\gamma^{\prime\prime}_z(1)&=& \Int D(z)+(2-q)B(z)
	-(\alpha+\beta-q) Q(z)\\[3ex]
	&=&
	\Int D(z) +(2-(\alpha+\beta))B(z)
	-(q-(\alpha+\beta))K(z).
	\end{array}
	\end{equation}
	holds true for any $z\in \N$ where was used identity \eqref{nehari}.
\end{rmk}

As a first step in order to obtain existence of solutions for the System \eqref{eq1} we shall prove that $J$ is coercive and bounded from below on $\N$. This result allow us to solve a minimization problem for the energy functional $J$ finding a ground state solution to the quasilinear elliptic Problem \eqref{eq1}.

{\begin{lem}\label{tec}
		Suppose that $(\phi_1)-(\phi_3)$ hold. Then there exist positive constants ${A_1}, R$ and $\theta_i$ in such way that
		\begin{itemize}
			\item[(i)]$\Int H(z)\geq  {A_1}||z||^{\theta_i}$,
			\item[(ii)]$\Int Q(z)\leq S^{\alpha+\beta}||z||^{\alpha+\beta},$
			\item[(iii)] $\Int K(z)\leq  	R||z||^q$
		\end{itemize}
		holds true for any $z \in W$.
	\end{lem}
	\begin{proof} Initially we shall prove the item $(i)$.  According to Proposition \ref{estimativanorma}
		we deduce that
		\begin{eqnarray*}
			\Int H(z)&\geq& \left(1-\Fr{m_1}{\alpha+\beta}\right)\Int \Phi_1(|\nabla u|)|+\left(1-\Fr{m_2}{\alpha+\beta}\right) \int_{\Omega}\Phi_2(|\nabla v|)\nonumber\\
			&\geq& \min\left\{1-\Fr{m_i}{\alpha+\beta}\right\}\Int A(z) \geq  A_1 \min\{||z||^{\min\{\ell_i\}},||z||^{\max\{m_i\}}\}=A_1 ||z||^{\theta_i}
		\end{eqnarray*}
		hold for some constants $A_1>0$ and $\theta_i\in\{\min\{\ell_i\},\max\{m_i\}\}$. This ends the proof of the item $(i)$.
		
		For the proof of item $(ii)$, we apply the Young's inequality and Sobolev embedding proving the following estimates
		\begin{equation}\label{desQ}
		\Int Q(z)\leq S^{\alpha+\beta}||z||^{\alpha+\beta}.
		\end{equation}
		For the proof of item $(iii)$ we use Sobolev embedding in order to prove that
		\begin{equation}\label{desK-1}
		\begin{array}{rcl}
		\Int K(z)&\leq&\mu ||a^+||_\infty\Int|u|^q+\lambda||c^+||_\infty\Int|v|^q\leq (\lambda +\mu)\bar{S}^q||z||^q
		\end{array}
		\end{equation}
		holds true for any $q > 1$. At the same time, for the singular case assuming that $0<q <1$, taking into account H\"older inequality and Sobolev embedding we infer that
		\begin{equation}\label{desK-2}
		\begin{array}{rcl}
		\Int K(z)&\leq&\mu \Int|u|^q+\lambda\Int|v|^q\leq  (\lambda +\mu)S_*(|\Omega|^{\frac{\ell_1^*-q}{\ell_1^*}}+|\Omega|^{\frac{\ell_2^*-q}{\ell_2^*}})||z||^q.
		\end{array}
		\end{equation}
		This finishes the proof.
	\end{proof}
	
	For the next result we consider some powerful estimates in order to get existence and multiplicity of solutions for the main Problem \eqref{eq1}. The main idea here is to consider some ideas discussed in Proposition \ref{fang}. Here for the functions $N$-functions $\Phi_1$ and $\Phi_2$ we shall consider the following result.
	
	\begin{lem}\label{lemA}
		Suppose that $(\phi_1)$-$(\phi_3)$ hold. Then we obtain the following estimates
		\begin{itemize}
			\item[(a)]$\min\left\{t^{\min\{\ell_i\}},t^{\max\{m_i\}}\right\}\Int B(z)\leq \Int B(tz) \leq \max\left\{t^{\min\{\ell_i\}},t^{\max\{m_i\}}\right\}\Int B(z);$\\
			\item[(b)]$\min\{\ell_i\}\min\left\{t^{\min\{\ell_i\}},t^{\max\{m_i\}}\right\}\Int B(z)\leq \Int B^\prime(tz)tz \leq\max\{m_i\} \max\left\{t^{\min\{\ell_i\}},t^{\max\{m_i\}}\right\}\Int B(z),$
		\end{itemize}
		for any $t>0$ and $z\in W, i=1,2$.
	\end{lem}
	As a consequence, we deduce that $J$ is coercive and bounded from below on the Nehari manifold. More precisely, we consider the following result
	\begin{prop}\label{coercive} Suppose that $(\phi_1)$-$(\phi_3)$ and $(H)$ hold.
		Then the energy functional $J$ is coercive and bounded from below on the Nehari manifold $\mathcal{N}_{\lambda,\mu}$.
	\end{prop}
	\begin{proof}  It is sufficient to see that
		$$J(z)= \Int \left(H(z)-\left(\Fr{1}{q}-\Fr{1}{\alpha+\beta}\right)K(z)\right) \geq A_1 ||z||^{\theta_i}-(\lambda+\mu)\left(\Fr{1}{q}-\Fr{1}{\alpha+\beta}\right)\bar{S}^{{q}}||z||^{q}$$
		holds true for any $z\in \N$ where $A_1>0$ and $\theta_i\in\{\min\{\ell_i\},\max\{m_i\}\}$. This concludes the proof.
	\end{proof}
	
	Now we shall prove that $\N$ is a $C^1$-manifold in the nonsingular case which is crucial in our arguments in order to get our main results. Here we shall apply the Lagrange Multiplier Theorem in order to solve a minimization problem. For nonsingular case or singular case we would like to mention that the sets $\mathcal{N}_{\lambda,\mu}^{-}$ and $\mathcal{N}_{\lambda,\mu}^{+}\cup \{0\}$ are also closed sets. These facts allows us to use the Ekeland's Variational Principle.
	
	\begin{lem}\label{c1} Suppose $(\phi_{1}) - (\phi_{3})$. Assume also that $q \in (0,1)$ or $q > 1$. Then there exists $\eta_{1} > 0$ small enough in such way that for any $(\lambda+\mu) \in (0, \eta_{1})$ we obtain
		\begin{enumerate}
			\item $\N^0=\emptyset$;
			\item $\N=\N^+\dot{\cup}\N^-$ is a $C^1$-manifold for any $q>1$;
			\item  The set $\overline{\N^+}$ is equals to $\N^+\cup \{0\}$. In particular, $\N^+\cup \{0\}$ is closed.
			\item The set $\N^-$ is closed.
		\end{enumerate}
	\end{lem}
	\begin{proof}
		First of all, we shall consider the proof for item (1). Arguing by contradiction we assume that $\N^0\neq\emptyset.$ Let $z\in\N^0$ be a fixed function. Clearly, we have $\gamma^\prime_z(1)=\gamma^{\prime\prime}_z(1)=0$. Now we mention that $\ell_i-2\leq \frac{\phi_i^\prime(t)t}{\phi_i(t)}$.
		Now taking into account $(\phi_3)$ ,\eqref{desQ} and the last assertion we obtain that
		\begin{eqnarray*}
			\Int ((\ell_1-q)\phi_1(|\nabla u| )|\nabla u|^2+(\ell_2-q)\phi_2(|\nabla v|)|\nabla v|^2)&\leq& \Int D(z)+(2-q)B(z)\nonumber\\
			&\leq&(\alpha+\beta-q)\Int Q(z)\nonumber\\
			&\leq& (\alpha+\beta-q)S^{\alpha+\beta}||z||^{\alpha+\beta}.
		\end{eqnarray*}
		
		On the other hand, using the Proposition \ref{estimativanorma}, there exists $A_1>0$, such that\\
		
		\begin{eqnarray}\label{est-theta}
		\Int ((\ell_1-q)\phi_1(|\nabla u| )|\nabla u|^2+(\ell_2-q)\phi_2(|\nabla v|)|\nabla v|^2) &\geq& \Int \ell_1(\ell_1-q)\Phi_1(|\nabla u|)+\ell_2(\ell_2-q)\Phi_2(|\nabla v|) \nonumber\\
		&\geq& A_1\min\{\ell_i(\ell_i-q)\}\min\{||z||^{\min\{\ell_i\}},||z||^{\max\{m_i\}}\}\nonumber\\
		&=& A_1\min\{\ell_i(\ell_i-q)\}||z||^{\theta_i}.
		\end{eqnarray}
		
		At this stage, using the estimates just above, we infer that
		
		$$A_1\min\{\ell_i(\ell_i-q)\}||z||^{\theta_i}\leq(\alpha+\beta-q)S^{\alpha+\mu}||z||^{\alpha+\beta}.$$
		Hence we know that
		$$||z||^{\alpha+\beta}\geq \Fr{A_1\min\{\ell_i(\ell_i-q)\}}{(\alpha+\beta-q)S^{\alpha+\beta}}||z||^{\theta_i}.
		$$
		These facts imply that
		\begin{equation}\label{des1}
		||z||\geq \left[\Fr{A_1\min\{\ell_i(\ell_i-q)}{(\alpha+\beta-q)S^{\alpha+\beta}}
		\right]^{\frac{1}{\alpha+\beta-\theta_i}}.
		\end{equation}
		
		On the other hand, using the hypothesis $(\phi_3)$,  \eqref{eq5} and taking into account either \eqref{desK-1} or \eqref{desK-2} we obtain
		$$\Int ((\alpha+\beta-m_1)\phi_1(|\nabla u| )|\nabla u|^2+(\alpha+\beta-m_2)\phi_2(|\nabla u| )|\nabla u|^2)\leq(\lambda+\mu)(\alpha+\beta-q)R||z||^{q},$$
		where $R$ is given by Lemma \ref{tec}-(iii).

		Using the same ideas discussed in \eqref{est-theta} we mention that
		$$A_1\min\{\ell_i(\alpha+\beta-m_i)\}\min\{||z||^{\min\{\ell_i\}},||z||^{\max\{m_i\}}\}\leq (\lambda+\mu)(\alpha+\beta-q)R||z||^{q}.$$
		Hence the last assertion says that
		$$\Fr{A_1\min\{\ell_i(\alpha+\beta-m_i)\}}{(\alpha+\beta-q)R}||z||^{\theta_i}
		\leq (\lambda+\mu)||z||^{q}.$$
		In this way, we observe that
		\begin{equation}\label{des2}
		\left[\Fr{A_1\min\{\ell_i(\alpha+\beta-m_i)}{(\alpha+\beta-q)R}\right]||z||^{\theta_i-q}\leq (\lambda+\mu).
		\end{equation}
		Under these conditions, using \eqref{des1} and \eqref{des2}, we get a contradiction for any
		\begin{equation}\label{eta1}
		(\lambda+\mu)<\min_{i=1,2}\left\{ \left[\Fr{A_1\min\{\ell_i(\ell_i-q)}{(\alpha+\beta-q)S^{\alpha+\beta}}\right]^{\frac{\theta_i-q}{\alpha+\beta-\theta_i}}\left[\Fr{A_1\min\{\ell_i(\alpha+\beta-m_i)}{(\el-q)R}\right]\right\}=:\eta_1.
		\end{equation}
		This finishes the proof for the item $(1)$.
		
		Now we shall prove the item $(2)$. Without any loss of generality, we take $z\in \N^+$. Define the function $G(z):=\left<J^\prime(z),z\right>, z \in \w$.  It is no hard to see that
		$$G^\prime(z)=J^{\prime\prime}(z)\cdot(z,z)+\left<J^\prime(z),z\right>=\gamma^{\prime\prime}_z(1)>0,\,\,\ \mbox{ for any } z\in\N^+.$$
		Hence, $0$ is a regular value for the functional $G$. As a consequence we see also that $\N^+$ is a $C^1$-manifold. Similarly, we can consider the proof for $\N^-$ proving that it is also a $C^1$-manifold. Therefore the proof of item $(2)$ follows due the fact that $\N^{0}=\emptyset$ for any $\lambda$ and $\mu$ in such way that $0 < \lambda + \mu$ is small enough. This completes the proof for the item $(2)$.
		
		At this moment we shall prove the item $(3)$. Since $\mathcal{N}^{0}_{\lambda,\mu}$ is empty the proof for the nonsingular case or singular case are the same.
		Let $(z_n)\subset \N^+$ be a sequence satisfying $z_n\rightarrow z$ in $W$. It is no hard to see that
		$$\lim_{n\rightarrow\infty} \gamma^\prime_{z_n}(1)=\gamma^\prime_{z}(1)=0.$$
		Hence $z \neq 0$ showing that $z\in \N$ or $z \equiv 0$. Assuming that $z \neq 0$ we obtain
		$\displaystyle\lim_{n\rightarrow\infty} \gamma^{\prime\prime}_{z_n}(1) = \gamma^{\prime\prime}_{z}(1)\geq 0$. As a consequence
		we know that $z\in \N^+$ or $z\equiv 0$. The last assertion ensures that $z\in \N^+\cup \{0\}$ proving that $\overline{\N^+} \subset \N^+\cup \{0\}$.
		
		On the other hand, we observe that $0\in\overline{\N^+}$. In fact, we mention that $J(z)\leq 0$ for any  $z\in \N^+$, see Proposition \ref{fib} ahead. Taking into account \eqref{eq2} and $(\phi_3)$ we also mention that
		\begin{eqnarray}\label{Jmenor0}
		\left(1-\min\left\{\frac{1}{\ell_i}\right\}\Fr{1}{\alpha+\beta}\right)\Int A(z)\leq \Int A(z) - \Fr{1}{\alpha+\beta} B(z)\leq \left(\Fr{1}{q}-\Fr{1}{\alpha +\beta}\right)\Int K(z).
		\end{eqnarray}
		Since $W$ is a reflexive Banach space, there exists $(z_n)\subset W$ in such way that $z_n\rightharpoonup 0$ in $W$ and $\|z_n\|=1$ for any $n\in \mathbb{N}$. Obviously, we obtain that $z_n \not\rightarrow 0$ in $W$. As a consequence we know that
		$$\liminf_{n\rightarrow\infty}\Int A(z_n)>0.$$
		Moreover, there exists $(t_n)\subset \mathbb{R}$ such that $t_nz_n\in \N^+$, see Proposition \ref{m_u-comp} ahead. Now, using \eqref{Jmenor0}, we infer that
		\begin{eqnarray}
		\min_i\{t_n^{\ell_i},t_n^{m_i}\}\left(1-\min\left\{\frac{1}{\ell_i}\right\}\Fr{1}{\alpha+\beta}\right)\Int A(z_n)&\leq& \left(1-\min\left\{\frac{1}{\ell_i}\right\}\Fr{1}{\alpha+\beta}\right)\Int A(t_nz_n)\nonumber\\
		&\leq&  t_n^q\left(\Fr{1}{q}-\Fr{1}{\alpha +\beta}\right)\Int K(z_n).\nonumber
		\end{eqnarray}
		Hence the last assertion implies that
		\begin{eqnarray}\label{t_n}
		\min_i\{t_n^{\ell_i-q},t_n^{m_i-q}\}\left(1-\min\left\{\frac{1}{\ell_i}\right\}\Fr{1}{\alpha+\beta}\right)\Int A(z_n)\leq\left(\Fr{1}{q}-\Fr{1}{\alpha +\beta}\right)\Int K(z_n).
		\end{eqnarray}
		Using the compact embedding $W \subset L^{q}(\Omega) \times L^{q}(\Omega)$ and \eqref{t_n} we deduce that $t_n\rightarrow 0$. Therefore, we obtain a sequence $(t_nz_n)\subset \N^+$ which satisfies $t_nz_n\rightarrow 0$ in $W$. As a consequence we obtain that $\overline{\N^+} = \N^+\cup \{0\}$. This ends the proof of item (3).

		Now we shall prove the item (4). Let $(z_n)\subset \N^-$ be a sequence satisfying $z_n\rightarrow z$ in $W$. It is no hard to see that
		$$\lim_{n\rightarrow\infty} \gamma^\prime_{z_n}(1) = \gamma^\prime_{z}(1)=0.$$
		Hence $z \neq 0$ showing that $z\in \N$ or $z \equiv 0$.
		Using the fact that $z_n\in \N^-$ and using the same ideas employed in \eqref{des1} there exists
		$C> 0$ in such way that
		\begin{equation*}
		C\leq \|z_n\|^{\alpha+\beta}.
		\end{equation*}
		Using the strong convergence we know that $C\leq \|z\|^{\alpha+\beta}$. As a consequence $z \neq 0$ which implies that
		$\lim_{n\rightarrow\infty} \gamma^{\prime\prime}_{z_n}(1) = \gamma^{\prime\prime}_{z}(1)\leq 0$. Since $z \neq 0$ the last assertion says that
		$z\in \N^-$. Thus $\N^-$ is a closed set proving the desired result. This ends the proof of Proposition \ref{c1}.
	\end{proof}

	Now we shall prove an auxiliary result using the Implicit Function Theorem in order to ensure the existence of a curve in the Nehari manifold.
	For related results we infer the reader to Yijing \cite{yijing2008}.
	
	\begin{lem}\label{proj}
		Given $(u,v)\in \N^-~(\N^+)$ there exist $\epsilon >0$ and a continuous function $f:B_\epsilon\rightarrow (0,\infty)$, where
		$B_\epsilon :=\{(w_1,w_2)\in W:\|(w_1,w_2)\|<\epsilon\}$, in such way that
		$$f(0,0)=1,~~f(w_1,w_2)(u+w_1,v+w_2)\in \N^-~~(\N^+),~~(w_1,w_2)\in B_\epsilon.$$
	\end{lem}
	\begin{proof}
		Define $F:\mathbb{R}\times W\rightarrow \mathbb{R}$ given by
		\begin{eqnarray*}
			F(t,w_1,w_2):=t^{-q}\Int B(t(u+w_1,v+w_2))
			- t^{\alpha+\beta-q}Q(u+w_1,v+w_2)-K(u+w_1,v+w_2).
		\end{eqnarray*}
		It is easy to see that
		\begin{eqnarray}
		\frac{\partial F}{\partial t}(1,0,0):= \int_{\Omega}(2-q)B(z)	+D(z)-(\alpha+\beta-q)Q(z)=\gamma_z^{\prime\prime}(1).\nonumber
		\end{eqnarray}
		Now we shall consider the proof of this proposition assuming that $(u,v)\in \N^-$. The proof for the Nehari manifold $(u,v)\in \N^+$ follows arguing in the same way. Using the fact that $(u,v)\in \N^-$ we obtain $\frac{\partial F}{\partial t}(1,0,0)<0$.  As a consequence, applying the Implicit Function Theorem for continuous functions, see for instance \cite[Remark 4.2.3]{Drabek0}, there exists $\epsilon >0$ and a continuous function $f:B_\epsilon \rightarrow (0,\infty)$ in such way that
		$$f(0,0)=1,~~f(w_1,w_2)(u+w_1,v+w_2)\in \N^-,~~w\in B_\epsilon.$$
		This finishes the proof.
	\end{proof}

	\section{Analysis of the Fibering Maps}\label{fib-map}
	
	In this section we give a complete description on the geometry for the fibering maps associated
	to the quasilinear elliptic System \eqref{eq1}. To the best of our knowledge, given $z\in W\setminus\{0\}$, the essential nature of fibering maps is
	determined taking into account the signs for the integrals $\Int K(z) \,\, \mbox{and} \,\, \Int Q(z).$

	Throughout this section is useful to consider the auxiliary functions $m_{z},\overline m_z: (0, \infty) \rightarrow \mathbb{R}$ of $C^{1}$ class defined by
	$$m_z(t)=t^{-q}\Int (B(tz)- Q(tz)) \qquad\mbox{and}\qquad \overline m_z(t)=t^{-\alpha-\beta}\Int (B(tz)- K(tz)), t > 0, z \in W\setminus\{0\}.$$
	Clearly, we see that
	\begin{equation*}
	\begin{array}{ccc}
	m^\prime_z(t) &=& t^{-(q+1)}\Int (2-q)B(tz)+D(tz)-(\alpha+\beta-q)Q(tz) \qquad\mbox{and}\\
	\overline m^\prime_z(t)&=& t^{-(\alpha+\beta+1)}\Int B^\prime(tz)tz-(\alpha+\beta)B(tz)-(q-\alpha-\beta)t^qK(z),
	\end{array}
	\end{equation*}
	where $t > 0$ and $z \in W\setminus\{0\}$.
	
	Now we shall consider a result comparing points $t z \in \N$ with the function $m_{z}$ and $\overline{m}_z$ . More precisely, we have the following interesting result
	\begin{lem}\label{m_uegamma_u}
		Let $t>0$ be fixed. Then $tz\in \N$ if and only if $t > 0$ is a solution for the following equations
		\begin{equation*}
		m_z(t)=\Int K(z)\qquad{or}\qquad\overline m_z(t)=\Int Q(z).
		\end{equation*}
	\end{lem}
	
	The next lemma is a powerful tool in order to get a precise information about the function $m_{z}$ and the fibering maps. More specifically, we shall consider the following result

	\begin{lem}\label{m_u-comp}
		\begin{enumerate} \item Suppose that $\Int Q(z)\leq 0$ holds. Then we obtain $\displaystyle m_{z}(0) := \lim_{t \rightarrow 0} m_{z}(t)=0 , m_{z}(\infty) :=  \lim_{t \rightarrow \infty} m_{z}(t) = \infty$ and $m^\prime_z(t)>0$ for any $t>0$.
			\item Suppose $\Int Q(z)>0$ and $(H)$. Then
			there exists an unique critical point for $m_{z}$, i.e, there is an unique point $\tilde{t}>0$ in such way that $m^\prime_z(\tilde{t})=0$.
			Furthermore, we know that $\tilde{t} > 0$ is a global maximum point for $m_{z}$ and $m_{z}(\infty) = - \infty$.
			\item Suppose $\Int K(z)>0$ $(H)$. Then
			there exists an unique critical point for $\overline m_{z}$, i.e, there is an unique point $\overline{t}>0$ in such way that $\overline m^\prime_z(\overline{t})=0$.
			Furthermore, we know that $\overline{t} > 0$ is a global maximum point for $\overline m_{z}$, $\overline m_{z}(0) = -\infty$, $\overline m_{z}(\infty) = 0$.
		\end{enumerate}
	\end{lem}
	\begin{proof}
		Initially, we shall prove the item $(1)$. The estimate \eqref{22} implies  that
		\begin{equation}\label{ee21}
		\min\{\ell_i-2\}(\phi_1(t)+\phi_2(t))\leq \phi_1^\prime(t)t+\phi_2^\prime(t)t\leq\\max\{m_i-2\}(\phi_1(t)+\phi_2(t)).
		\end{equation}
		As a consequence we see that
		$$
		\begin{array}{rcl}
		m^\prime_z(t)&\geq& t^{-(q+1)} \Int\min(\ell_i-q)B(tz)-(\alpha+\beta-q)Q(tz)>0.
		\end{array}
		$$
		Hence the function $m_z$ is increasing for any $t>0,$ i.e, we have $m^{\prime}_{z}(t) > 0$ for any $t > 0$.
		Moreover, we shall prove that $m_{z}(0) = 0$. In fact, using Lemma \ref{lemA},
		we deduce that
		\begin{equation}\label{ae}
		\Int t^{\max\{m_i\}-q} B(z)-t^{(\alpha+\beta)-q} Q(z)\leq m_z(t)\leq \Int t^{\min\{\ell_i\}-q} B(z)-t^{(\alpha+\beta)-q}Q(z),
		\end{equation}
		where $ t \in [0, 1].$ Taking the limit in estimate \eqref{ae} we get $\displaystyle \lim_{t \rightarrow 0^+} m_{z}(t) = 0$.
		Furthermore, we also mention that Lemma \ref{lemA} implies that
		\begin{equation}
		m_{z}(t) \geq   \int_{\Omega} t^{\min\{\ell_i\}-q}B(z) - t^{\alpha+\beta-q} Q(z),\  t \geq 1.\nonumber
		\end{equation}
		Due the fact that $\ell_i > q $ the last assertion implies that $m_{z}(\infty) =  \infty$.
		This finishes the proof of item $(1)$.
		
		Now we shall prove the item $(2)$.  As first step we mention that $m_z$ is increasing for $ t > 0$ small enough and $\ds\lim_{t\to\infty}m_z(t)=-\infty$. More specifically, for $0<t<1$ and using one more time \eqref{ee21} and Lemma \ref{lemA} we get
		$$\begin{array}{rcl}
		m^\prime_z(t)&\geq&
		\Int \min(\ell_i-q)t^{-(q+1)} B(tz)-(\alpha+\beta-q)t^{-(q+1)}Q(tz)\\[2ex]
		&\geq & \dfrac{1}{t}\Int\min(\ell_i-q)t^{\max\{m_i\}-q} B(z)-(\alpha+\beta-q)t^{\alpha+\beta-q} Q(z).
		\end{array}$$
		Since $m_i<\alpha+\beta,~i=1,2,$ we mention that $m^\prime_z(t)>0$ for any $t>0$ small enough. Furthermore, arguing as above we see also that
		if $t>1$, $$m_z(t)\leq \Int t^{\max\{m_i\}-q}B(z)-t^{\alpha+\beta-q}Q(z).$$
		Therefore, we have $\ds\lim_{t\to\infty}m_z(t)=-\infty$ where was used the fact that $m_i<\alpha+\beta,~i=1,2$.
		
		Now the main goal in this proof is to show that $m_z$ has an unique critical point $\tilde{t}>0.$ Note that, we have $m^{\prime}_z(t)=0$ if and only if $$(2-q)t^{-(\alpha+\beta)}\Int B(tz)+t^{-(\alpha+\beta)}\Int D(tz)=(\alpha+\beta-q)\Int Q(z).$$
		Define the auxiliary function $\eta_{z}: (0,\infty) \rightarrow \mathbb{R}$ given by
		$$\eta_z(t)=(2-q)t^{-(\alpha+\beta)}\Int B(tz)+t^{-(\alpha+\beta)}\Int D(tz).$$
		Here we emphasize that
		\begin{equation}\label{ee3}
		\ds\lim_{t\to 0^+}\eta_z(t)=+\infty.
		\end{equation}
		Indeed, using Lemma \ref{lemA} and \eqref{ee21} and putting $0<t<1$, we easily see that
		$$\begin{array}{rcl}\eta_z(t)&\geq&\min(\ell_i-q)t^{-(\alpha+\beta)}\Int B(tz)\\[2ex]
		&\geq& \min(\ell_i-q)t^{\max\{m_i\}-(\alpha+\beta)} \Int B(z).
		\end{array}$$
		Using one more time that $m_i < \alpha+\beta$ and $\ell_i > q $ it follows that \eqref{ee3} holds true.
		
		On the other hand, we mention that $\eta_{z}$ is a decreasing function which satisfies
		\begin{equation}\label{ee33}
		\lim_{t\to \infty}\eta_z(t)=0.
		\end{equation}
		In fact, using one more time Lemma \ref{lemA} and \eqref{ee21}, for any $t>1$, we observe that
		\begin{equation}\label{ee4}
		\min(\ell_i-q)t^{\min\{\ell_i\}-(\alpha+\beta)}\Int B(z)\leq \eta_z(t)\leq \max(m_i-q)t^{\max\{m_i\}-(\alpha+\beta)}\Int B(z).
		\end{equation}
		Hence \eqref{ee4} says that \eqref{ee33} holds true. Moreover, we have that
		\begin{eqnarray}
		\eta_z^\prime(t)&=&\Int (2-(\alpha+\beta))(2-q)t^{-(1+\alpha+\beta)} B(tz)+ (5-(\alpha+\beta+q))t^{-(1+\alpha+\beta)} D(tz) \nonumber\\
		&+&t^{-(1+\alpha+\beta)}\Int \phi_1^{\prime\prime}(t|\nabla u|)|\nabla tu|^4+\phi_2^{\prime\prime}(t|\nabla v|)|\nabla tv|^4.\nonumber
		\end{eqnarray}
As a consequence, using the estimates in Remark \ref{conseqphi3}, we obtain that
		$$
		\begin{array}{rcl}
		\eta_z^{\prime}(t)&\leq& ((2-(\alpha+\beta))(2-q)+m_1-2)t^{-(1+\alpha+\beta)} \Int \phi_1(|\nabla tu|)|\nabla tu|^2\\[2ex]
		&+& ((2-(\alpha+\beta))(2-q)+m_2-2)t^{-(1+\alpha+\beta)} \Int \phi_2(|\nabla tv|)|\nabla tv|^2\\[2ex]
		&+& ((m_1+1)-(\alpha+\beta+q))t^{-(1+\alpha+\beta)}\Int\phi_1^\prime(t|\nabla u|)|\nabla tu|^3\\[2ex]
		&+&((m_2+1)-(\alpha+\beta+q))t^{-(1+\alpha+\beta)}\Int\phi_2^\prime(t|\nabla v|)|\nabla tv|^3
		\end{array}
		$$
		
		It follows from $(m_i-1)\ell_i\leq(m_i-1)m_i\leq \max\{(m_i-1)m_i\}$ and hypothesis $(H)$ that
		\begin{equation}\label{antigaH}
		{0}<q<\Fr{\ell_i(\alpha+\beta-m_i)}{\alpha+\beta-\ell_i}.
		\end{equation}
		
		As a consequence, we also see that
		\begin{equation*}
		(\alpha+\beta-1)(m_i-\ell_i)<(\alpha+\beta-\ell_i)(m_i-q).
		\end{equation*}
		Moreover, we mention that $((2-\alpha+\beta)(2-q)+m_i-2)+((m_i+1)-(\alpha+\beta+q))(\ell_i-2) = (\alpha+\beta-1)(m_i-\ell_i)-(\alpha+\beta-\ell_i)(m_i-q)$.
		Under these conditions and using $(\phi_3)$ it is no hard to verify that
		$$
		\begin{array}{c}
		\eta_z^\prime(t)\leq\max[(\alpha+\beta-1)(m_i-\ell_i) - (\alpha+\beta-\ell_i)(m_i-q)]t^{-(1+\alpha+\beta)}\Int B(tz)
		<0.
		\end{array}
		$$
		Thus we conclude that $\eta_z$ is decreasing function proving that $m_z $ has an unique critical point which is a maximum critical point for $m_{z}$.
		
		Now we shall prove the item $(3)$. Here we borrow some ideas discussed in the proof of item $(3)$. At this point, we mention that $\displaystyle\lim_{t\rightarrow 0_+}\overline m_z(t)=-\infty, \lim_{t\rightarrow \infty}\overline m_z(t)=0$ and $\overline{m}_z$ is increasing for $ t > 0$ small enough. Namely, using the Lemma \ref{lemA} we observe that
		\begin{eqnarray}
		\lim_{t\rightarrow 0_+}\overline m_z(t)\leq \lim_{t\rightarrow 0_+}\Int t^{-(\alpha+\beta)}\max \{t^{\ell_1},t^{\ell_2}\}B(z)-t^{q-\alpha+\beta}K(z)=-\infty,\nonumber\\
		\lim_{t\rightarrow +\infty} |\overline m_z(t)|\leq \lim_{t\rightarrow \infty}\Int t^{-(\alpha+\beta)}\max \{t^{m_1},t^{m_2}\}B(z)+t^{q-\alpha+\beta}K(z)=0.\nonumber
		\end{eqnarray}
		Here was used the fact that $\Int K(z)>0$ and $q<\ell_i\leq m_i<\alpha+\beta,~i=1,2$. Furthermore, using the Lemma \ref{lemA}, we see that
		\begin{eqnarray}
		\overline{m}^\prime_z(t)\geq t^{-(\alpha+\beta+1)}\Int (\min\{\ell_i\}-\alpha-\beta)t^{\max\{m_i\}}B(z)+(\alpha+\beta-q)K(z)\nonumber
		\end{eqnarray}
		holds true for $ t > 0$ small enough. As a consequence $\overline{m}^\prime_z(t)>0$ for any $t>0$ small enough.
		
		Now the main point is to show that $\overline m_z$ has an unique critical point $\overline{t}>0$. Note that, $\overline m^{\prime}_z(t)=0$ if and only if
		\begin{eqnarray}\label{eta-bar}
		\Int t^{-q}[{(\alpha+\beta)}B(tz)-B^\prime(tz)tz]=(\alpha+\beta-q)\Int K(z).
		\end{eqnarray}
		Define the auxiliary function $\overline\eta_{z}: (0,\infty) \rightarrow \mathbb{R}$ given by
		\begin{eqnarray}
		\overline\eta_z(t)=\Int t^{-q}[(\alpha+\beta)B(tz)-B^\prime(tz)tz].\nonumber
		\end{eqnarray}
		It follows from Lemma \ref{lemA} that
		$$ t^{\max\{m_i\}-q}(\alpha+\beta-\max\{m_i\})\Int B(z)\leq\overline\eta_z(t)\leq t^{\min\{\ell_i\}-q}(\alpha+\beta-\min\{m_i\})\Int B(z),~0<t<1.$$
		Consequently, we have that $\displaystyle\lim_{t\rightarrow 0^+}\overline\eta_{z}(t)=0$.
		Moreover, using Lemma \ref{lemA}, we infer also that
		$$t^{\min\{\ell_i\}-q}(\alpha+\beta-\max\{m_i\})\Int B(z)\leq\overline\eta_z(t),$$
		Hence we obtain that $\displaystyle\lim_{t\rightarrow \infty}\overline\eta_{z}(t)=\infty$.
		Now we claim that $\overline\eta_{z}$ is a increasing function. In fact,
		in view of the hypothesis (H), we mention that
		\begin{eqnarray*}
			\overline\eta^\prime_{z}(t)&=&t^{-1-q}[-q(\alpha+\beta)B(tz)+(\alpha+\beta+q-1)B^\prime(tz)tz-B^{\prime\prime}(tz)t^2z^2]\nonumber\\
			&\geq& t^{-1-q}[-q(\alpha+\beta)+(\alpha+\beta+q-1)\min\{\ell_i\}-\max\{(m_i-1)m_i\}]B(tz)>0.
		\end{eqnarray*}
		Thus, applying \eqref{eta-bar}, there exists an unique critical point $\overline t>0$ which is a global maximum point for $\overline m_{z}$.
		The proof for this lemma is now complete.
	\end{proof}
	
	Now we shall prove that $m_{z}$ has a behavior at infinity and at the origin which are described by the sign of $\displaystyle \int_{\Omega} K(x) \,\, \mbox{and} \,\, \displaystyle \int_{\Omega} Q(z).$
	This is crucial tool in to prove a complete description on the geometry for the fibering maps. In order to perform our next results we shall consider the functions $g_{\theta_i}:[0,\infty)\rightarrow \mathbb{R},~\theta_i\in\{\max\{m_i\},\min\{\ell_i\}\}$ defined by
	\begin{equation*}
	g_{_{\theta_i}}(t):=t^{\theta_i-1}\Int B(z)-t^{\el-1}\Int Q(z),\ \ t>0.
	\end{equation*}
	It is easy to see that there exists $\bar{t}:=\overline{t}_{\theta_i}>0$ such that
	$$g_{_{\theta_i}}(\overline{t}_{\theta_i})=\displaystyle \max_{t>0}g_{_{\theta_i}}(t).$$
	Actually, we observe that
	\begin{equation*}
	\bar{t}:=\overline{t}_{\theta_i}
	=\left[\Fr{(\theta_i-1)\Int B(z)}{(\alpha+\beta-1)\Int Q(z)}\right]^{\frac{1}{\alpha+\beta-\theta_i}}.
	\end{equation*}
	Inspired in part by the recent works \cite{CarvalhoClaudiney,tarantello,tarantellonewmann} we shall assume the following assumptions:
	\begin{itemize}
		\item[$(D)$] Suppose that either $\overline{t}_{\min\{\ell_i\}},\overline{t}_{\max\{m_i\}}\geq 1$ or $\overline{t}_{\min\{\ell_i\}},\overline{t}_{\max\{m_i\}}\leq 1$. Then we obtain
		$$(\lambda+\mu)<\min_{i=1,2}\left\{ \left[\Fr{A_1\min\{\ell_i(\ell_i-q)}{(\alpha+\beta-q)S^{\alpha+\beta}}\right]^{\frac{\theta_i-q}{\alpha+\beta-\theta_i}}
		\left[\Fr{A_1\min\{\ell_i(\alpha+\beta-m_i)}{(\el-q)R}\right]\right\}=:\eta_1.$$
		
		\item [$(E)$] Suppose $\ell_i<m_i$ and $\overline{t}_{\min\{\ell_i\}}\leq 1\leq \overline {t}_{\max\{m_i\}}$ hold. Assume also that $$(\lambda+\mu)\leq \min\left\{\eta_i,\Fr{\alpha+\beta- \{m_i\}}{\{m_i\}-1}\right\}.$$
	\end{itemize}
	
	In order to find a second solution for the quasilinear elliptic System \eqref{eq1} we consider a more restrictive condition which can be written in the following form:
	\begin{itemize}
		\item [$(E)'$] Suppose that $\overline{t}_{\min\{\ell_i\}}\leq 1\leq \overline {t}_{\max\{m_i\}}$ holds. Assume also that
		$$(\lambda+\mu)\leq \min\left\{\frac{q}{m_i}\eta_1,\Fr{\alpha+\beta-m_i}{m_i-1}\right\}.$$
	\end{itemize}
	
	\begin{rmk}
		Notice that $g_{\max\{m_i\}}(t)=g_{\min\{\ell_i\}}(t)=m_z(t)$ if and only if $t=1$.
	\end{rmk}
	\begin{lem}\label{mapf1f2}
		Suppose either $(D)$ or $(E)$. Then we obtain the following identity
		\begin{eqnarray}\label{des-f}
		\max_{t>0}m_z(t)\geq \Int K(z),\ \ z\in W.
		\end{eqnarray}
	\end{lem}
	\begin{proof} Firstly, we mention that $\displaystyle\max_{t>0}m_z(t)>0$. It is important to see that if $\displaystyle\int_\Omega K(z) \leq 0$ implies that \eqref{des-f} is satisfied. In this way, we can consider the case $\displaystyle\int_\Omega K(z) > 0$. Here we shall consider the hypothesis $(D)$. Assuming the hypotheses $(E)$ or $(E)'$ the proof can be done using similar ideas discussed in the present proof.
		
		Let us consider the case $\overline{t}_{\min\{\ell_i\}}, \ \overline{t}_{\max\{m_i\}}\geq 1$. The proof for the other cases are analogous which can be found in \cite{CarvalhoClaudiney}. Remembering that $\overline{t}_{\min\{\ell_i\}}, \ \overline{t}_{\max\{m_i\}}\geq 1$ and using the fact that
		$\overline{t}_{\min\{\ell_i\}}\geq 1$ we can proceed as in \cite{CarvalhoClaudiney,tarantello,tarantellonewmann} proving the following inequalities
		\begin{equation*}
		(\alpha+\mu-1)\Int B(z)\leq(\min\{\ell_i\}-1)\Int K(z)
		\end{equation*}
		and
		\begin{equation}\label{A1A2}
		A_1\min{\ell_i}||z||^{\min{\ell_i}}\leq\Int B(z)\leq A_2\max\{m_i\}||z||^{\max\{m_i\}}.
		\end{equation}
		
		Now, we observe that
		\begin{equation*}
		\begin{array}{rcl}
		g(\bar{t}_{\theta_i})&=&\bar{t}^{\theta_i-q}\left[\Int (K(z)-\bar{t}^{\alpha+\beta-\theta_i}Q(z))\right] = \bar{t}^{\theta_i-q}\Fr{\alpha+\beta-\theta_i}{\alpha+\beta-q}\Int B(z)\\[3ex]
		&=&\left[\Fr{\theta_i-q}{(\alpha+\beta-q)\Int Q(z)}\right]^{\frac{\theta_i-q}{\alpha+\beta-\theta_i}}\left[\Fr{\alpha+\beta-\theta_i}{\alpha+\beta-q}\right]\left[\Int B(z)\right]^{\frac{\alpha+\beta-q}{\alpha+\beta-\theta_i}}
		\end{array}
		\end{equation*}
		
		Since $\displaystyle\max_{t>0}m_z(t)\geq \displaystyle\max_{t>0}g_{_{\min\{\ell_i\}}}(t)=\displaystyle\max_{t>0}g_{_{\min\{\ell_i\}}}(\bar{t}_{\min\{\ell_i\}})$, we observe also that
		\begin{equation*}\label{A2}
		\begin{array}{rcl}
		\displaystyle\max_{t>0} m_z(t)&\geq& \left[\Fr{\min\{\ell_i\}-q}{(\alpha+\beta-q)\Int Q(z)}\right]^{\frac{\min\{\ell_i\}-q}{\alpha+\beta-\min\{\ell_i\}}}\left[\Fr{\alpha+\beta-\min\{\ell_i\}}{\alpha+\beta-q}\right]\left[\Int B(z)\right]^{\frac{\alpha+\beta-q}{\alpha+\beta-\min\{\ell_i\}}}.\end{array}
		\end{equation*}
		
		At this stage, taking into account \eqref{A1A2}, we deduce that
		\begin{equation*}\label{A2}
		\begin{array}{rcl}
		\displaystyle\max_{t>0} m_z(t)&\geq& \left[\Fr{A_1\min\{\ell_i\}(\min\{\ell_i\}-q)}{(\alpha+\beta-q)\Int Q(z)}\right]^{\frac{\min\{\ell_i\}-q}{\alpha+\beta-\min\{\ell_i\}}}\left[\Fr{A_1\min\{\ell_i\}(\alpha+\beta-\min\{\ell_i\})}{\alpha+\beta-q}\right]||z||^{\frac{\min\{\ell_i\}(\alpha+\beta-q)}{\alpha+\beta-\min\{\ell_i\}}}.\end{array}
		\end{equation*}
		Furthermore, using Lemma \ref{tec}, we see that
		\begin{equation*}\label{A2}
		\begin{array}{rcl}
		\displaystyle\max_{t>0} m_z(t)&\geq& \left[\Fr{A_1\min\{\ell_i\}(\min\{\ell_i\}-q)}{(\alpha+\beta-q)||z||^{\alpha+\beta}S^{\alpha+\beta}}\right]^{\frac{\min\{\ell_i\}-q}{\alpha+\beta-\min\{\ell_i\}}}\left[\Fr{A_1\min\{\ell_i\}(\alpha+\beta-\min\{\ell_i\})}{\alpha+\beta-q}\right]||z||^q||z||^{\frac{(\min\{\ell_i\}-q)(\alpha+\beta)}{\alpha+\beta-\min\{\ell_i\}}}\\[3ex]
		&\geq&
		\left[\Fr{A_1\min\{\ell_i\}(\min\{\ell_i\}-q)}{(\alpha+\beta-q)||z||^{\alpha+\beta}S^{\alpha+\beta}}\right]^{\frac{\min\{\ell_i\}-q}{\alpha+\beta-\min\{\ell_i\}}}\left[\Fr{A_1\min\{\ell_i\}(\alpha+\beta-\min\{\ell_i\})}{\alpha+\beta-q}\right]\Fr{\Int K(z)}{\bar{S}^q(\lambda+\mu)}\\[3ex]
		&\geq&
		\left[\Fr{A_1\min\{\ell_i\}(\min\{\ell_i\}-q)}{(\alpha+\beta-q)S^{\alpha+\beta}}\right]^{\frac{\min\{\ell_i\}-q}{\alpha+\beta-\min\{\ell_i\}}}\left[\Fr{A_1\min\{\ell_i\}(\alpha+\beta-
			\min\{m_i\})}{\alpha+\beta-q}\right]\Fr{\Int K(z)}{\bar{S}^q(\lambda+\mu)}
		.\end{array}
		\end{equation*}
		As a consequence, applying \eqref{eta1}, we obtain the desired result.  This ends the proof.
	\end{proof}
	
	Since we stay interesting in quasilinear linear elliptic systems with indefinite nonlinearities the fibering maps geometry is not simple. Depending on the signs for the concave and convex terms we prove that the fibering has critical points. This is contained in the following result
	
	\begin{lem}\label{fib}
		Let $z\in W\setminus\{0\}$ be a fixed function. Then we shall consider the following assertions:
		\begin{enumerate}
			\item Assume that $\Int Q(z)\leq 0$. Then $\gamma^\prime_z(t)\neq 0$ for any $t>0$ and $(\lambda+\mu) > 0$ whenever $\Int K(z) \leq 0$. Furthermore, there exist an unique $ t_1= t_1(z,\lambda,\mu)$ in such way that $\gamma^\prime_z(t_1)=0$  and $ t_1 z\in \N^+$ whenever $\Int K(z) > 0.$
			
			\item Assume that $\Int Q(z)> 0$ holds. Then there exists an unique $t_1=t_1(z,\lambda,\mu)> \tilde t$ such that
			$\gamma^\prime_z(t_1)= 0$ and $t_1 z\in \N^-$ whenever $\Int K(z)\leq 0$.
			\item Assume that $(H)$ holds. For each $\mu, \lambda$, such that $(\mu +\lambda)>0$ is small enough there exist unique $0<t_1=t_1(z,\lambda)<\tilde t, \overline t<t_2=t_2(z,\lambda,\mu)$, where $\tilde t, \overline t$ were defined in the Lemma \ref{m_u-comp}, such that 	$\gamma^\prime_z(t_1)=\gamma^\prime_z(t_2)=0$, $ t_1 z \in \N^+$ and $ t_2 z \in \N^-$ whenever $\Int K(z)> 0$, $\Int Q(z)> 0$ holds.
		\end{enumerate}	
	\end{lem}
	\begin{proof}
		First of all, we shall consider the proof for the case $\Int Q(z)\leq 0$ and $\Int K(z)\leq 0$. Using Lemma \ref{m_u-comp} (1) it is easy to verify that
		\begin{equation}\label{m_u}
		m_z(0)=0,~\lim_{t\rightarrow\infty}m_z(t) = \infty \,\,\mbox{and} \,\, m^\prime_z(t)>0, t \geq 0.\nonumber
		\end{equation}
		Under these conditions we deduce that
		$$
		m_z(t)\neq \Int K(z) \,\, \mbox{for any} \,\,t>0, \lambda,\mu>0.
		$$
		According to Lemma \ref{m_uegamma_u} we deduce that $tz\not\in \N$ for any  $t>0$. In particular, we see also that $\gamma_z^\prime(t)\neq 0$ for each $t>0$.
		
		Now we shall consider the proof for the case $\Int K(z) > 0$ and $\Int Q(z) \leq 0$. Using one more time Lemma \ref{m_u-comp} (1) we observe that $m_{z}(0)= 0 , m_{z}(\infty) = \infty$ and $m_{z}$ is a increasing function. In particular, the equation
		$$m_z(t)= \Int K(z)$$
		admits exactly one solution $t_1= t_1(z,\lambda,\mu)>0$. Hence, using Lemma \ref{m_uegamma_u}, we know that $t_1 z\in \N$ proving that $\gamma^\prime_z( t_1) =0$. Furthermore, using the identity
		\begin{equation}\label{rel-m-u-gamma-u}
		m_z(t)=t^{1-q}\gamma_z^\prime(t)+ \Int K(z),\nonumber
		\end{equation}
		we easily see that
		$$0<m^\prime_z(t_1) = t_1^{1-q}\gamma^{\prime\prime}_z(t_1).$$
		In particular, we have been proven that $t_1z\in\N^+$.
		
		Now we shall consider the proof for the case $\Int K(z) \leq 0$ and $\Int Q(z) > 0$. Here the function $m_{z}$
		admits an unique turning point $\tilde{t} > 0$, i.e, we have that  $m^{\prime}_{z}(t) = 0,~t > 0$ if only if $t = \tilde{t}$, see Lemma \ref{m_u-comp} (2).
		Moreover, $\tilde{t}$ is a global maximum point for $m_{z}$ in such way that $m_{z}(\tilde{t}) > 0, m_{z}(\infty) = - \infty$.
		As a product there exists an unique $t_1>\tilde t$ such that
		$$m_z(t_1)= \Int K(z).$$
		Here we emphasize that $m_z^\prime(t_1)<0$ where we have used the fact that $m_{z}$ is a decreasing function in $(\tilde{t}, \infty)$. As a consequence we obtain $0>m^\prime_z(t_1)=t_1^{1-q}\gamma_z^{\prime\prime}(t_1)$ proving that $t_1z\in\N^-$.
		
		At this moment we shall consider the proof for the case $\Int K(z) > 0$ and $\Int Q(z) > 0$. In view of Lemma \ref{mapf1f2} we can consider $(\lambda+\mu)>0$
		small enough in such way that
		$$m_z(\tilde t) > \Int K(z)\qquad\mbox{and}\qquad \overline m_z(\overline t) > \Int Q(z).$$

		It is worthwhile to mention that $m_{z}$ is increasing in $(0, \tilde{t})$ and decreasing in $(\tilde{t}, \infty)$, and  $\overline m_{z}$ is increasing in $(0, \overline{t})$ and decreasing in $(\overline{t}, \infty)$.
		It is not hard to verify that there exist exactly two points $0<t_1=t_1(z,\lambda,\mu)  <\tilde t,\overline t < t_2=t_2(z,\lambda,\mu)$ such that
		$$m_z(t_i)=\Int K(z)\qquad\mbox{and}\qquad \overline m_z( t_i) = \Int Q(z),~i=1,2.$$
		Additionally, we have that $m_z^\prime(t_1),\overline m_z^\prime(t_1)>0$ and $m_z^\prime(t_2),\overline m_z^\prime (t_2)<0$. Arguing as in the previous step we ensure that $t_1z\in\N^+$ and $t_2z\in\N^-$. This completes the proof.  \hfill\cqd
		\vskip.2cm
		
		The next lemma shows that for any $\lambda+\mu>0$ small enough  the function $\gamma_z$ assumes  positive values. This is crucial for the proof of our
		main theorems proving that $\gamma_z$ admits one or two critical points.
		Now we shall prove that $J$ is away from zero on the Nehari manifold $\N^{-}$. In particular, any critical point on
		$\N^{-}$ provide us a nontrivial critical point.
		\vskip.2cm
		\begin{lem}\label{nehari-} There exist $\delta_1, \eta_{2} > 0$ in such way that $J(z)\geq\delta_1$ for any $z\in \N^{-}$
			where $0<\lambda +\mu< \eta_2$.
		\end{lem}
		\proof Since $z\in\N^-$ we infer that
		$\gamma^{\prime\prime}_{z}(1)<0$. Arguing as in the proof of Lemma \ref{c1} we mention that
		$$||z||^{\alpha+\beta}\geq \left[\Fr{A_1\min\ell_i(\ell_i-q)}{(\alpha+\beta-q)S^{\alpha+\beta}}\right]||z||^{\theta_i}$$
		where $\theta_i$ and $A_1$ are given by Proposition \ref{estimativanorma}.
		
		Putting all these facts together we see that
		\begin{equation*}
		||z||\geq \left[\Fr{A_1\min\ell_i(\ell_i-q)}
		{(\alpha+\beta-q)S^{\alpha+\beta}}\right]^{\frac{1}{\alpha+\beta-\theta_i}}.
		\end{equation*}
		Furthermore, arguing as in the proof of Proposition \ref{coercive}, we also mention that
		$$\begin{array}{rcl}
		J(z)&\geq& A_1 ||z||^{\theta_i}-(\lambda +\mu)\left(\Fr{1}{\alpha+\beta}-\Fr{1}{q}\right)R||z||^q\\
		&\geq&\left[\Fr{A_1\min\{\ell_i(\ell_i-q)\}}{(\alpha+\beta-q)S^{\alpha+\beta}}\right]^{\frac{q}{\alpha+\beta-\theta_i}}\left[\min\left\{\ell_i\left(\Fr{1}{m_i}-\Fr{1}{\alpha+\beta}\right)\right\}\left[\Fr{A_1\min\{\ell_i(\ell_i-q)\}}{(\alpha+\beta-q)S^{\alpha+\beta}}\right]^{\frac{\theta_i-q}{\alpha+\beta-\theta_i}}\right.\\
		&-&\left.(\lambda +\mu)\left(\Fr{1}{\alpha+\beta}-\Fr{1}{q}\right)R\right]
		=\left[\bar{A_i} +(\lambda+\mu) B \right],
		\end{array}$$
		where $A_1$, $R$ were defined in Lemma \ref{tec}-(iii) and Proposition \ref{estimativanorma}, respectively. Here we also define
		$$\bar{A_i}:=\left[\Fr{A_1\min\{\ell_i(\ell_i-q)\}}{(\alpha+\beta-q)S^{\alpha+\beta}}\right]^{\frac{q}{\alpha+\beta-\theta_i}}\min\left\{\ell_i\left(\Fr{1}{m_i}-\Fr{1}{\alpha+\beta}\right)\right\}\left[\Fr{A_1\min\{\ell_i(\ell_i-q)\}}{(\alpha+\beta-q)S^{\alpha+\beta}}\right]^{\frac{\theta_i-q}{\alpha+\beta-\theta_i}}$$
		and
		$$B=\left[\Fr{A_1\min\{\ell_i(\ell_i-q)\}}{(\alpha+\beta-q)S^{\alpha+\beta}}\right]^{\frac{q}{\alpha+\beta-\theta_i}}\left(\Fr{1}{q}-\Fr{1}{\alpha+\beta}\right)R.$$
		This concludes the proof whenever $\lambda+\mu<\Fr{\bar{A_i}}{B}$, i.e, the proof for this proposition follows for any $$\lambda+\mu<\min\left\{\Fr{q}{m_i}\right\}\eta_1$$ where $\eta_1>0$ is given by \eqref{eta1}.\hfill\cqd

		Now we shall prove that any minimizer on $\N^{+}$ has negative energy. More specifically,
		we can show the following result
		
		\begin{lem}\label{nehari+} Suppose $(H)$. Then there exist $z \in \N^{+}$ and $\eta_{1} > 0$ in such way that $\ds\inf_{z\in\N^+} J(z)  \leq J(z) < 0$ for each $0<\lambda +\mu< \eta_1$. In particular, we obtain $\ds\inf_{z\in\N^+} J(z) =\ds\inf_{z\in\N} J(z) $ for each $0<\lambda+\mu < \eta_1$.
		\end{lem}
		\proof Let $z\in \N^+$ be fixed. Here we observe that $\gamma_z^{\prime\prime}(1)>0$. As a consequence we mention that
		$$
		\begin{array}{rcl}
		\Int Q(z)&<&\frac{1}{\alpha+\beta-q}\Int D(z)+(2-q)B(z)
		\\[2ex]
		&\leq&\displaystyle\frac{\max(m_i-q)}{\alpha+\beta-q}\Int B(z).
		\end{array}$$
		
		On the other hand, using the inequality just above and hypothesis $(\phi_3)$ (cf. Remark \ref{conseqphi3}), we easily see that
		\begin{eqnarray*}
			J(z)&\leq& \left(\Fr{1}{\ell_1}-\Fr{1}{q}\right)\Int \phi_1(|\nabla u|)|\nabla u|^2+\left(\Fr{1}{\ell_2}-\Fr{1}{q}\right)\Int \phi_1(|\nabla v|)|\nabla v|^2+\left(\Fr{1}{q}-\Fr{1}{\alpha+\beta}\right)\Int Q(z) \nonumber \\
			&<&\Fr{1}{q}\left[\Fr{q-\ell_1}{\ell_1 }+\Fr{m_1-q}{\alpha+\beta }\right]\Int \phi_1(|\nabla u|)|\nabla u|^2+\Fr{1}{q}\left[\Fr{q-\ell_2}{\ell_2 }+\Fr{m_2-q}{\alpha+\beta }\right]\Int \phi_2(|\nabla v|)|\nabla v|^2 \nonumber \\
			&\leq& \Fr{1}{q}\max\left\{\left[\Fr{q-\ell_i}{\ell_i }+\Fr{m_i-q}{\alpha+\beta }\right]\right\}\Int B(z)
			. \nonumber \\
		\end{eqnarray*}
		In view of the hypothesis \eqref{antigaH} we obtain that right side in the last inequality is negative. As a consequence we observe that $$\ds\inf_{z\in\N^+} J(z)\leq J(z) <0.$$
		In addition, we stress out that $\N=\N^-\cup\N^+$ since $(\lambda+\mu)<\eta_1$ where $\eta_1$ is given by \eqref{eta1} and $\ds\inf_{z\in\N^-} J(z)>0$. Hence we deduce that $$\ds\inf_{z\in\N^+} J(z)=\ds\inf_{z\in\N} J(z).$$
		This completes the proof.
	\end{proof}
	
	\begin{lem}\label{criticalpoint}
		Suppose $(\phi_{1}) - (\phi_{3})$ and $q > 1$. Let $z$  be a local minimum (or local maximum) for $J$ in $\N$. Then
		$z$ is a critical point of $J$ on $W$ for any $(\lambda+\mu)<\eta_1$.
	\end{lem}
	\begin{proof}
		The proof follows using the same ideas discussed in \cite{MLCarvalhocritico,JME}. Here we omit the details.
	\end{proof}
	
	\begin{lem}\label{no-semitriv}
		Suppose $(\phi_{1}) - (\phi_{3})$ and $0<q<1$ or $q > 1$. Let $z\in \N^+$ a weak solution of System \eqref{eq1}. Then $z$ is not semitrivial, that is, $z \neq (u, 0)$ and $z \neq (0,v)$ with $u, v \in W_{0}^{1,\Phi}(\Omega)$.
	\end{lem}
	\begin{proof}
		Without any loss generality we assume that $v\equiv 0$. Here we mention that $u$ is a non zero solution of problem  
		\begin{equation}\label{aux-2}
				\left\{
				\begin{array}{l}
				-\Delta_{\Phi_1} v =\mu a(x)|v|^{q-2}v \mbox{ in }\    \Omega, \\
				v=0,~\mbox{on}~\partial \Omega.
				\end{array}
				\right.
		\end{equation}
		Hence
		\begin{eqnarray}\label{u0}
		\Int \phi_1(|\nabla u|)|\nabla u|^2=\lambda\Int a(x)|u|^q>0.
		\end{eqnarray}
		Define $\widetilde{\Omega}:=\{x\in\Omega: b(x)>0\}$ and taking into account conditions $(A)$ or $(B)$ we obtain $|\widetilde{\Omega}|>0$. Now, consider the problem
		\begin{equation}\label{aux-4}
		\left\{
		\begin{array}{l}
		-\Delta_{\Phi_2} w =\mu c(x)|w|^{q-2}w \mbox{ in }\    \widetilde\Omega, \\
		w=0,~\mbox{on}~\partial \widetilde\Omega.
		\end{array}
		\right.
		\end{equation}
		For the singular case, taking into account the hypothesis $(C)$ and $0 < q < 1$, the Problem \eqref{aux-4} has a solution $w_2\in W^{1,\Phi_1}_0(\widetilde\Omega)$ which can be obtained in Goncalves et al  \cite[Thm. 2.1(i)]{Carvalho}. In the nonsingular case the same existence result still holds assuming that $1<q\leq \ell$, see \cite{GoncalveCarvalho}. As a consequence
		$$\int_{\widetilde{\Omega}} \phi_2(|\nabla w_2|)|\nabla w_2|^2=\mu\int_{\widetilde{\Omega}} c(x)|w_2|^{q}.$$
		Now, taking $w_2=0$ in $\Omega/\widetilde{\Omega}$ we infer that
		\begin{eqnarray}\label{w_2}
		\int_{{\Omega}} \phi_2(|\nabla w_2|)|\nabla w_2|^2=\mu\int_{{\Omega}} c(x)|w_2|^{q}.
		\end{eqnarray}
		Consequently, we know that
		\begin{eqnarray}\label{Q>0}
		\Int Q(u,w_2)=\Int b(x)|u|^\alpha|w_2|^{\beta}>0.
		\end{eqnarray}
		According to \eqref{u0} and \eqref{w_2} we also obtain
		\begin{eqnarray}\label{BK}
		\Int B(u,w_2)=\Int K(u,w_2).
		\end{eqnarray}
		Furthermore, using Lemma \ref{fib}, there exists $0<t_1<\overline{t}=\overline{t}(u,w_2)$ in such way that $(t_1u,t_1w_2)\in \N^+$ satisfying
		$$J(t_1u,t_1w_2)=\displaystyle\inf_{0<t\leq \overline{t}}J(tu,tw_2).$$
		
		Now we claim that $\overline t>1$. In fact, we mention that $\overline{t}$ satisfies \eqref{eta-bar}. Using Lemma \ref{lemA} we deduce that
		$$(\alpha+\beta-q)\Int K(u,w_2)\leq (\alpha+\beta-\min\{\ell_i\})\min\left\{\overline{t}^{\min\{\ell_i\}-q},\overline{t}^{\max\{m_i\}-q}\right\}\Int B(u,w_2).$$
		Using \eqref{BK} and $q<\min\{\ell_i\}<\alpha+\beta$ we obtain $\overline{t}>1$. This ends the proof for the claim.

		The last assertion and \eqref{Q>0} imply that
		$$J(t_1u,t_1w_2)\leq J(u,w_2)<J(u,0)=\inf_{z\in \N^+}J(z)=\inf_{z\in \N}J(z).$$
		This is a contradiction. Hence the weak solution z is not semitrivial, i.e, we have that $z\neq (u,0)$. Analogously way we can show that $z\neq(0,v)$. This finishes the proof of this proposition. 	
	\end{proof}

	\section{The proof of our main theorems}\label{main-res}
	
	Now we shall consider the proof of our main results. Now, we borrow some ideas discussed in \cite{Brown2}.
	As a first step we shall consider an auxiliary result in the following form:
	
	\begin{prop}\label{strong-converg}
		Suppose $(\phi_{1}) - (\phi_{3})$ and $0<q < 1$ or $q>1$. Let $(z_n) \in \mathcal{N}_{\lambda,\mu}^{+}$ be a minimizer sequence such that $z_n\rightharpoonup z$ in $W$.
		Then $z_n\rightarrow z$ and $z \in \mathcal{N}_{\lambda,\mu}^{+}$ for all $\lambda+\mu<\eta_1$.
	\end{prop}
	\begin{proof}	
		In fact, up to a subsequence we have
		$$\begin{array}{rcl}
		z_n&\to& z \,\,\mbox{a.e.}\,\, \Omega,\\
		z_n&\to& z  \,\,\mbox{in}\,\, L_{\Phi_1}(\Omega)\times L_{\Phi_2}(\Omega).
		\end{array} $$
		As a consequence, using the compact embeddings $W\hookrightarrow (L^{\alpha+\beta}(\Omega))^2$ and $W\hookrightarrow (L^q(\Omega))^2$, it follows that
		$$
		\Int K(z_n)\to \Int K(z)\, \mbox{and} \,\Int Q(z_n)\to \Int Q(z).$$
		Furthermore, using the fact that $z_n\in\N^+$, we also obtain
		\begin{equation*}
		\begin{array}{rcl}
		\Int K(z_n)&=&\Fr{q(\alpha+\beta)}{\alpha+\beta-q}\Int H(z_n)-J(z_n)\Fr{q(\alpha+\beta)}{\alpha+\beta-q}\\[3ex]
		&\geq& \Fr{q(\alpha+\beta)}{\alpha+\beta-q}\left(1-\Fr{\max\{m_i\}}{\alpha+\beta}\right)\Int A(z_n)-J(z_n)\Fr{q(\alpha+\beta)}{\alpha+\beta-q}\\[3ex]
		&\geq& -J(z_n)\Fr{q(\alpha+\beta)}{\alpha+\beta-q.}
		\end{array}
		\end{equation*}
		As a consequence we mention that
		$$\Int K(z)\geq -\ds \Fr{q(\alpha+\beta)}{\alpha+\beta-q}\inf_{z\in \N^+}J(z) >0 \,\, \mbox{ and } \,\, z \,\, \mbox{is not zero}.$$
		Taking into account  Lemma \ref{fib} there exists $t_1>0$ in such way that $t_1z\in\N^+, \gamma^\prime(t_1)=0 \,\, \mbox{and}\,\, \gamma(t_1)=J(t_1z)<0$.
		Arguing by contradiction we assume that $z_n$ does not converge to $z$ in $W$. Using the same ideas explored in \cite{MLCarvalhocritico,JME} we infer that
		\begin{equation}\label{liminf0}
		\begin{array}{c}
		\Int A(z)< \liminf_{n\rightarrow \infty}\Int A(z_n)\quad{and}\quad\Int B(z)<\liminf_{n\rightarrow \infty} \Int B(z_n).
		\end{array}
		\end{equation}
		At this moment since $t_1z\in\N^+$ we also mention that
		$$0=\gamma_z^\prime(t_1)=t_1^{-1}\Int (B(t_1z)- K(t_1z) - Q(tz)).$$
		Using \eqref{liminf0} we observe that
		$$0=\gamma^\prime(t_1)<\liminf_{n\rightarrow \infty}\left(t_1^{-1}\Int (B(t_1z_n)- K(t_1z_n) - Q(tz_n) ) \right)=\liminf_{n\rightarrow \infty}\gamma_{z_n}^\prime(t_1).$$
		As a consequence there exists $n_0\in\mathbb{N}$ in such way that
		\begin{eqnarray}\label{gamma-z_n-t_1}
		\liminf_{n\rightarrow \infty}\gamma_{z_n}^\prime(t_1)>0, \ \ \ \forall n>n_0.
		\end{eqnarray}
		Using one more time that $(z_n)\subset\N^+$ and applying  Lemma \ref{fib} we obtain that $\gamma_{z_n}^\prime(t)<0$ for any $t\in(0,1)$ and $\gamma^\prime_{z_n}(1)=0$.
		Here, we observe that from \eqref{gamma-z_n-t_1} we conclude that $t_1>1$.
		
		On the other hand, using $t_1z\in \N^+,~t_1>1$ and \eqref{liminf0}, we deduce that
		$$J(t_1z)\leq J(z)<\liminf_{n\rightarrow \infty} J(z_n)=\ds\inf_{z\in\N^+} J(z).$$
		This is an absurd showing that $z_n$ converges to $z$ in $W$. This ends the proof.	
	\end{proof}

	\subsection{The first weak solution for the nonsingular case:}
	Here we emphasize that $q>1$. Now we stay in position in order to prove that any critical point for $J$ on $\N$ is a free critical point, i.e, is a critical point in the whole space $W$. According to Proposition \ref{coercive} we know that $J$ is coercive and bounded from below in $\N^+$. Let $z_n=(u_n,v_n)$ be a minimizer sequence for $J$ in $\N^+$. It is easy to see that $(z_n)$ is bounded in $W$. Up to a subsequence there exists $z\in W$ such that
	$z_n\rightharpoonup z \,\, \mbox{in } \,\, W. $
	It follows from the Proposition \ref{strong-converg} that $z_n\rightarrow z ~\mbox{in}~W.$
	In addition, the last assertion says also that
	$$J(z)=\lim_{n\rightarrow \infty} J(z_n)=\ds\inf_{z\in\N^+} J(z).$$
	Hence, applying Lemma \ref{criticalpoint}, we have that $z$ is a weak solution to the quasilinear elliptic System \eqref{eq1}. Since $J(z)=J(|z|)$ and $|z|=(|u|,|v|) \in \N^+$, we assume that $z$ is a nonnegative solution to the elliptic System \eqref{eq1}. It follows from Lemma
	\ref{no-semitriv} that $u,v\neq0$. This ends the proof.
	\hfill\cqd
	\vskip.2cm
	
	\subsection{The first weak solution for the singular case:}
	Since we are interesting in the case $q \in (0,1)$ the energy functional is not in $C^1$ class. However, using the Nehari method, we stay in position to find existence and multiplicity of solutions for the System \eqref{eq1} taking into account the behavior of the fibering maps.  The next result can be stated in the following way
	
	\begin{lem}\label{lambda-dist-0}
		Suppose $(\phi_{1}) - (\phi_{3})$ and $0<q < 1$. Let $(z_n) \in \mathcal{N}_{\lambda,\mu}^{+}$ be a minimizer sequence such that $z_n\rightharpoonup z$ in $W$. Then there exist $\eta_1>0$ and $c_{\lambda,\mu}>0$ such that  $\gamma^{\prime\prime}_{z_n}(1)>c_{\lambda,\mu}$ for all $\lambda+\mu<\eta_1$.
	\end{lem}
	\begin{proof}
		Arguing by contradiction we assume that set $\gamma^{\prime\prime}_{z_n}(1)=o_n(1)$. Using the same ideas explored in \eqref{des1} we obtain
		\begin{equation}\label{des11}
		||z_n||\geq \left[\Fr{A_1\min\{\ell_i(\ell_i-q)}{(\alpha+\beta-q)S^{\alpha+\beta}}\right]^{\frac{1}{\alpha+\beta-\theta_i}}+o_n(1).
		\end{equation}
		
		On the other hand, taking into account \eqref{eq5}, Remark \ref{conseqphi3}, \eqref{desK-1} or \eqref{desK-2}, we use the H\"older inequality (for Orlicz-Sobolev space), we obtaining the following estimates
		\begin{eqnarray}
		\Fr{A_1\min\{\ell_i(\alpha+\beta-m_i)\}}{(\alpha+\beta-q)R}||z_n||^{\theta_i} &\leq &\Int ((\alpha+\beta-m_1)\phi_1(|\nabla u_n| )|\nabla u_n|^2+(\alpha+\beta-m_2)\phi_2(|\nabla u_n| )|\nabla u_n|^2)\nonumber\\
		&\leq&(\lambda+\mu)(\alpha+\beta-q)R||z_n||^{q} +o_n(1).\nonumber
		\end{eqnarray}
		Here we emphasize that $R$ is given by Lemma \ref{tec}-(iii) where $0<q<1$. Hence, using the last assertion together with \eqref{des11} we obtain that
		\begin{equation}\label{des12}
		\Fr{A_1\min\{\ell_i(\alpha+\beta-m_i)\}}{(\alpha+\beta-q)R}||z_n||^{\theta_i-q}
		\leq (\lambda+\mu)+{o_n(1)}.
		\end{equation}
		Under these conditions, using \eqref{des11} and \eqref{des12}, we get a contradiction for any $\lambda$ and $\mu$ satisfying
		\begin{equation*}
		(\lambda+\mu)<\min_{i=1,2}\left\{ \left[\Fr{A_1\min\{\ell_i(\ell_i-q)}{(\alpha+\beta-q)S^{\alpha+\beta}}\right]^
		{\frac{\theta_i-q}{\alpha+\beta-\theta_i}}\left[\Fr{A_1\min\{\ell_i(\alpha+\beta-m_i)}{(\el-q)R}\right]\right\}=:\eta_1.
		\end{equation*}
		Here we recall that $\eta_1>0$ is given by Lemma \ref{c1}. This ends the proof.
	\end{proof}
	\begin{rmk}\label{limites}
		In the proof of Theorem \ref{teorema2} we shall use the following facts: Let $(z_n)\subset \N$ be a sequence, $w=(\varphi_1,\varphi_2)\in W$ and  $g_n,R_i:(0,\infty)\rightarrow \mathbb{R},~i=1,2,3,4,$ are functions given by
		$$R_1(t):=\int_\Omega B(g_n(t)(z_n+tw))-B(z_n),\qquad R_2(t):=\int_\Omega [g_n(t)]^{q}K(z_n+tw)-K(z_n),$$
		$$R_3(t):=\int_\Omega[g_n(t)]^{\alpha+\beta}Q((z_n+tw))-Q(z_n),\qquad R_4(t):=\int_{\Omega}A(z_n)-A(g_n(t) (z_n+tw)),$$
		and  $g_n(t):=f_n(tw),$ where $f_n$ was defined in Lemma \ref{proj}. Then we obtain the following limits:
		\begin{enumerate}
			\item It holds that
			\begin{eqnarray*}
				\lim_{t\rightarrow 0_+}\frac{R_1(t)}{t}= g_{n}^\prime(0)\int_{\Omega}D(z_n)+2B(z_n)+C(z_n)w+A^\prime(z_n)w
			\end{eqnarray*}
			\nd Here was used the derivative $t\mapsto \int_{\Omega}B(g_n(t)(z_n+tw_n))$ at the origin;
			\item $\displaystyle\lim_{t\rightarrow 0_+}\frac{R_2(t)}{t}=qg^\prime_{n}(0)\Int K(z_n)+\liminf_{t\rightarrow 0^+} \int_{\Omega}\left[\lambda a(x)\Fr{|u_n+t \varphi_1|^q-|u_n|^q}{t}+\mu c(x)\Fr{|v_n+t \varphi_2|^q-|v_n|^q}{t}\right]$;
			\item $\displaystyle\lim_{t\rightarrow 0_+}\frac{R_3(t)}{t}=\int_\Omega(\alpha+\beta) g_{n}^\prime(0)Q(z_n)+Q^\prime(z_n)w$;
			\item $\displaystyle\lim_{t\rightarrow 0_+}\frac{R_4(t)}{t} = -\int_\Omega g_n^\prime(0)B(z_n)+A^\prime(z_n)w$.
		\end{enumerate}
		Here, we observe that $g^\prime_{n}(0)\in[-\infty,\infty]$ is understood as the right derivative of $g_n$ at $t=0$.	
	\end{rmk}
	From now on, we shall apply Ekeland's variational principle in order to find a minimizer for the energy functional $J$. The Ekeland's variational principle implies that there exists a minimizer sequence $(z_n) \in \mathcal{N}_{\lambda,\mu}^{+}$ satisfying
	\begin{itemize}
		\item[(i)] $J(z_n)< \displaystyle\inf_{\N^+}J+\frac{1}{n}$;
		\item[(ii)] $J(w)\geq J(z_n)-\frac{1}{n}\|w-z_n\|,~\forall w\in \N^+$.
	\end{itemize}
	The main idea here is to apply Lemma \ref{proj}. In order to do that we take $g_n(t):=f_n(tw),~w=(\varphi_1,\varphi_2)$ and $t>0$ small enough such that  $g_n(t)(z_n+tw)\in \N^+$. Notice that $g_n(0)=1$. Using the last assertion we see that
	\begin{eqnarray}\label{nehari1}
	0=\gamma^\prime_{g_n(t)(z_n+tw)}(1)-\gamma^\prime_{z_n}(1)=R_1(t)-R_2(t)-R_3(t),
	\end{eqnarray}
	where $R_i,~i=1,2,3$, were defined in the Remark \ref{limites}.
	Putting all these fact together with \eqref{nehari1} we also have
	\begin{eqnarray}\label{g_n-1}
	0&\leq& g^\prime_n(0)\left\{\int_{\Omega} D(z_n)+(2-q)B(z_n)-(\alpha+\beta-q)Q(z_n) \right\}+\int_{\Omega}C(z_n)w+A^\prime(z_n)w-Q^\prime(z_n)w \nonumber \\
	&=&g^\prime_n(0)\gamma_{z_n}^{\prime\prime}(1)+\int_{\Omega}C(z_n)w+A^\prime(z_n)w-Q^\prime(z_n). 
	\end{eqnarray}
	Without any loss of generality we assume that $g_n^\prime(0)$ is well defined, see \cite{yijing}.
	\vskip.2cm
	\nd {\bf Claim:} $g_n^\prime(0)\neq \pm\infty$. In fact, using $z_n\in\N^+$, it follows from \eqref{g_n-1} that $g_n^\prime(0)\neq -\infty$.
	
	Now, we shall prove that $g_n^\prime(0)\neq \infty$. The proof for this claim follows arguing by contradiction. Assuming that $g_n^\prime(0)=\infty$ we obtain for each $t>0$ small enough that $g_n(t)>1$. Therefore, taking into account the item $(ii)$ from Ekeland's variational principle, we also infer that
	\begin{eqnarray}\label{g_n-2}
	[g_n(t)-1]\frac{\|z_n\|}{n}+tg_n(t)\frac{\|\varphi\|}{n}&\geq& J(z_n)-J(g_n(t)(z_n+tw))=R_4(t)-\Fr{1}{q}R_2(t)-\Fr{1}{\alpha+\beta}R_3(t).
	\end{eqnarray}
	As a consequence, taking the limit $t\rightarrow 0_+$ and using the fact that $z_n\in \N$, we obtain
	\begin{eqnarray}\label{Des-11}
	\frac{\|\varphi\|}{n}&\geq& \frac{g^\prime_{n}(0)}{q}\left[\gamma_{z_n}^{\prime\prime}(1)-q\frac{\|z_n\|}{n}\right]-\frac{1}{q}\int_\Omega(q-1)A^\prime(z_n)w-C(z_n)w+
	\frac{\alpha+\beta-q}{\alpha+\beta}Q^\prime(z_n)w.
	\end{eqnarray}
	Here, we were used the terms $\displaystyle\lim_{t\rightarrow 0^+}\frac{R_i(t)}{t},~i=2,3,4,$ given in the Remark \ref{limites}. Due to the  Lemma \ref{lambda-dist-0} we have that $\gamma_{z_n}^{\prime\prime}(1)> c_{_{\lambda,\mu}}>0$. So, the assertion \eqref{Des-11} is impossible if $g_n^\prime(0)=\infty$. This proves the claim just above.

	At this stage we shall prove that $z$  is in $\N^+$. Moreover, we mention that $z$ is a weak solution to the quasilinear elliptic System \eqref{eq1}.
	First of all, using {(ii)} from Ekeland's variational principle and \eqref{g_n-2}, we infer that
	\begin{eqnarray}\label{cg_n-2}
	[g_n(t)-1]\frac{\|z_n\|}{n}+tg_n(t)\frac{\|\varphi\|}{n}&\geq& J(z_n)-J(g_n(t)(z_n+tw))=R_4(t)-\Fr{1}{q}R_2(t)-\Fr{1}{\alpha+\beta}R_3(t).
	\end{eqnarray}
	In this way, dividing \eqref{cg_n-2} by $t>0$, taking the limit $t\to 0^+$ and using  the terms $\displaystyle\lim_{t\to 0^+}\Fr{R_2(t)}{t}$ and $\displaystyle\lim_{t\to 0^+}\Fr{R_4(t)}{t}$ which were obtained in the Remark \ref{limites}, we also see that
	\begin{equation*}
	\begin{array}{rcl}
	\Fr{1}{n}\left[g^\prime_{n}(0)||z_n||+||\varphi||\right]&\geq& -g^\prime_{n}(0)\Int [B(z_n)-K(z_n)-Q(z_n)]+\int_\Omega\Fr{1}{\alpha+\beta}Q^\prime(z_n)w-A^\prime(z_n)w\\[2ex]
	&+&\Fr{1}{q}\liminf_{t\rightarrow 0^+}\int_{\Omega} \left[\lambda a(x)\Fr{|u_n+t \varphi_1|^q-|u_n|^q}{t}+\mu c(x)\Fr{|v_n+t \varphi_2|^q-|v_n|^q}{t}\right].
	\end{array}
	\end{equation*}
	Now, using the fact that $z_n\in \N$, we mention that
	\begin{equation}\label{liminf1}
	\begin{array}{rcl}
	\Fr{1}{n}\left[g^\prime_{n}(0)||z_n||+||\varphi||\right]&\geq&
	\displaystyle\int_\Omega\Fr{1}{\alpha+\beta}Q^\prime(z_n)w-A^\prime(z_n)w\\
	&+&\Fr{1}{q}\liminf_{t\rightarrow 0^+} \int_{\Omega}\left[\lambda a(x)\Fr{|u_n+t \varphi_1|^q-|u_n|^q}{t}+\mu c(x)\Fr{|v_n+t \varphi_2|^q-|v_n|^q}{t}\right].\\
	\end{array}
	\end{equation}
	In addition, we also mention that
	$$\left[\lambda a(x)\Fr{|u_n+t \varphi_1|^q-|u_n|^q}{t}+\mu c(x)\Fr{|v_n+t \varphi_2|^q-|v_n|^q}{t}\right]\geq 0$$
	holds for any $x\in \Omega$ and $t>0$. Using the Fatou's Lemma it follows that
	\begin{equation}\label{fatou}
	\frac 1 q\Int K^\prime(z_n)w\leq \Fr{1}{q}\liminf_{t\rightarrow 0^+} \int_{\Omega}\left[\lambda a(x)\Fr{|u_n+t \varphi_1|^q-|u_n|^q}{t}+\mu c(x)\Fr{|v_n+t \varphi_2|^q-|v_n|^q}{t}\right].
	\end{equation}
	Under these conditions, using \eqref{liminf1} and \eqref{fatou}, taking into account also that $(z_n)$ and $(g^\prime_{n}(0))$ are bounded sequences, we deduce that
	\begin{equation*}\begin{array}{rcl}
	\displaystyle\frac 1 q\Int K^\prime(z_n)w&\leq &\Fr{1}{n}\left[g^\prime_{n}(0)||z_n||+||\varphi||\right]+\Int A^\prime(z_n)w-\Fr{1}{\alpha+\beta}Q^\prime(z_n)w\\[2ex]
	&\leq&\Fr{(C+||\varphi||)}{n}+\Int A^\prime(z_n)w-\Fr{1}{\alpha+\beta}Q^\prime(z_n)w
	\end{array}
	\end{equation*}
	It follows from Proposition \ref{strong-converg} that $z_n\rightarrow z$. Hence we infer also that
	$$\begin{array}{rcl}
	\displaystyle\frac{1}{q}\liminf_{n\to\infty}\Int K^\prime(z_n)w &\leq& \Int A^\prime(z)w
	-\Fr{1}{\alpha+\beta}\Int Q^\prime(z)w.
	\end{array}$$
	Using one more time Fatou's Lemma we mention that
	$$\begin{array}{rcl}
	\displaystyle\frac{1}{q}\Int K^\prime(z)w &\leq& \Int A^\prime(z)w
	-\Fr{1}{\alpha+\beta}\Int Q^\prime(z)w
	\end{array}$$
	Notice that, for each $\varphi=(e_1,e_2)\in W\cap (C^1(\Omega)\times C^1(\Omega))$ satisfying $\varphi >0$, we obtain
	$$\Fr{1}{q}\Int K^\prime(z)\varphi<\infty.$$
	Hence $z>0$ a.e. in $\Omega$. As a consequence $0<z\in W$ satisfies the following estimate
	\begin{equation}\label{cdes1}
	\begin{array}{rcl}
	\displaystyle \Int A^\prime(z)\varphi-\Fr{1}{q}\Int K^\prime(z)\varphi
	-\Fr{1}{\alpha+\beta}\Int Q^\prime(z)\varphi\geq 0,
	\end{array}
	\end{equation}
	holds true $\forall \varphi\in W$ satisfying $\varphi>0$.
	
	It remains to prove that $z$ is a nonnegative weak solution for the quasilinear elliptic System \eqref{eq1}. The main to tool here is to consider \eqref{cdes1} together the ideas discussed in \cite{yijing2011,yijing2008}. Consider $\varphi=(\varphi_1,\varphi_2)\in W$ be a fixed function and $\epsilon>0$. Define $\psi=(\psi_1,\psi_2)\in W,\ \psi\geq 0$ given by
	$$\psi=(\psi_1,\psi_2):=([u+\epsilon\varphi_1]^+,[v+\epsilon\varphi_2]^+).$$
	For our purpose we need to consider the set
	$$\Omega_\epsilon:=\{x\in \Omega: z+\epsilon\varphi>0\}.$$
	Using $\psi$ as a test function in \eqref{cdes1} together with the fact that $z\in\N,$ we infer that
	$$\begin{array}{rcl}
	0&\leq&\displaystyle\int_{\Omega_\epsilon} A^\prime(z)(z+\epsilon\varphi)-\displaystyle\frac{1}{q}\int_{\Omega_\epsilon}K^\prime(z)(z+\epsilon\varphi)
	-\Fr{1}{\alpha+\beta}\displaystyle\int_{\Omega_\epsilon} Q^\prime(z)(z+\epsilon\varphi)\\[4ex]
	&=&\left(\Int-\displaystyle\int_{\Omega\setminus\Omega_\epsilon} \right)\left[ A^\prime(z)(z+\epsilon\varphi)
	-
	\Fr{1}{q}K^\prime(z)(z+\epsilon\varphi)
	-\Fr{1}{\alpha+\beta} Q^\prime(z)(z+\epsilon\varphi)\right].
	\end{array}$$
	
	As a consequence, using the fact that $z$ belongs to the Nehari manifold, we also have
	$$\begin{array}{c}
	0\leq \epsilon \Int\left[ A^\prime(z)\varphi-\Fr{1}{q} K\prime(z)\varphi
	-\Fr{1}{\alpha+\beta} Q^\prime(z)\varphi\right]-\displaystyle\int_{\Omega\setminus\Omega_\epsilon}\left[A^\prime(z)(z+\epsilon\varphi)
	-
	\displaystyle\frac{1}{q}  K^\prime(z)(z+\epsilon\varphi)-\Fr{1}{\alpha+\beta} Q^\prime(z)(z+\epsilon\varphi)\right]
	\end{array}$$
	Consequently, the last estimates imply that
	$$\begin{array}{rcl}
	0&\leq& \epsilon \Int \left[A^\prime(z)\varphi-\Fr{1}{q} K^\prime(z)\varphi
	-\Fr{1}{\alpha+\beta} Q^\prime(z)\varphi\right]-\epsilon\displaystyle\int_{\Omega\setminus\Omega_\epsilon} \left[A^\prime(z)\varphi-\Fr{1}{q} K^\prime(z)\varphi
	-\Fr{1}{\alpha+\beta} Q^\prime(z)\varphi\right].
	\end{array}$$
	It is worthwhile to mention that in $\displaystyle\Omega\setminus\Omega_\epsilon$ we obtain $\varphi_1, \varphi_2<0$. This fact implies also that
	$$\begin{array}{rcl}
	0&\leq& \epsilon \Int \left[A^\prime(z)\varphi-\frac{1}{q}K^\prime(z)\varphi-\Fr{1}{\alpha+\beta} Q^\prime(z)\varphi\right]
	-\epsilon\int_{\Omega\setminus\Omega_\epsilon} A^\prime(z)\varphi.
	
	\\[4ex]
	\end{array}$$
	At this moment we observe that $\lim_{\epsilon\to 0^+} |\{\Omega\setminus\Omega_\epsilon\}|=0$. As a product we obtain
	$$\displaystyle\lim_{\epsilon\to 0^+}\displaystyle\int_{\Omega\setminus\Omega_\epsilon} A^\prime(z)\varphi=0.$$
	Now, dividing the last expression by $\epsilon>0$ and taking the limit we also obtain that
	$$\begin{array}{rcl}
	\Int \left[A^\prime(z)\varphi-\frac{1}{q}K^\prime(z)\varphi-\frac{1}{\alpha+\beta}Q^\prime(z)\varphi\right]\geq 0
	\end{array}.
	$$
	At this moment, using the test function $-\varphi$ instead of $\varphi$, we infer that
	$$\begin{array}{rcl}
	\Int A^\prime(z)\varphi &=&\Int \left[\frac{1}{q}K^\prime(z)\varphi+\frac{1}{\alpha+\beta}Q^\prime(z)\varphi\right]
	\end{array}.
	$$
	In other words, we have been showed that $z$  is a positive weak solution to the elliptic System \eqref{eq1}. In particular, we have also that $z\in \N$. Actually, we also mention that $z\in \N^+$ which can be proved arguing by contradiction.
	
	Now we observe that $\displaystyle\lim_{\lambda,\mu\to 0^+}||z_{\lambda, \mu}||=0$. Indeed, since $z_{\lambda,\mu}\in\N^+$ and arguing as in the proof of Lemma \ref{c1} we infer that
	\begin{equation*}
	||z_{\lambda, \mu}||^{\theta_i-q} \leq \Fr{(\lambda+\mu)}{\rho} \,\, \mbox{where} \,\, \rho =\dfrac{A_1\min\{\ell_i(\alpha+\beta-m_i)}{(\alpha+\beta-q)R}.
	\end{equation*}
The same argument works also in the nonsingular case, that is,  $\displaystyle\lim_{\lambda,\mu\to 0^+}||z_{\lambda, \mu}||=0$ holds true both in nonsingular case and singular case. This is a powerful tool to consider the asymptotic behavior for the solutions
	$z_{\lambda, \mu}$ which is used in the proof of our main results. 
	
	\subsection{Some proprieties for the singular and nonsingular case}
	
	\begin{lem}\label{no-semitriv-}
		Suppose $(\phi_{1}) - (\phi_{3})$ and $0<q<1$ or $q > 1$. Let $z\in \N^-$ a weak solution of System \eqref{eq1}. Then $z$ is not the weak semitrivial solution.
	\end{lem}
	\begin{proof}
		Note that $z$ is a solution of System \eqref{eq1}. Furthermore, using Lemma \ref{nehari-}, we obtain that $J(z)>\delta_1>0$.
		Now we claim that $z$ is not semitrivial, i.e, we have that $z \neq (u, 0)$. In fact, arguing by contradiction, we assume that $z=(u,0)$. As a consequence $z$ is a weak solution to the elliptic Problem \eqref{aux-2}. Hence $0<\Int \phi_1(|\nabla u|)|\nabla u|^2=\Int a(x)|u|^q$ proving that
		\begin{eqnarray}
		J(u,0)&=&\Int \Phi_1(|\nabla u|)-\frac{1}{q}a(x)|u|^q\nonumber\\
		&=& \Int \Phi_1(|\nabla u|)-\frac{1}{q}\phi_1(|\nabla u|)|\nabla u|^2\nonumber\\
		&\leq &\left(1-\frac{\ell_1}{q}\right)\Int \Phi_1(|\nabla u|)<0.\nonumber
		\end{eqnarray}
		This is a contradiction. Similarly, we see also that $z\neq (0,v)$.
	\end{proof}

	\begin{prop}\label{strong-converg-}
		Suppose $(\phi_{1}) - (\phi_{3})$ and $0<q < 1$ or $q>1$. Let $(z_n) \in \mathcal{N}_{\lambda,\mu}^{-}$ be a minimizer sequence such that $z_n\rightharpoonup z$ in $W$.
		Then $z_n\rightarrow z$ and $z \in \mathcal{N}_{\lambda,\mu}^{-}$ for all $\lambda+\mu<\eta_1$.
	\end{prop}
	\begin{proof}	
		Notice that, using Lemma \ref{nehari-}, there exists $\delta_1>0$ such that $J(z)\geq \delta_1$ for any $z\in \N^{-}$. As a consequence
		$$J{^-}:= \ds\inf_{z\in \N^{-}}J(z)\geq \delta_1>0.$$
		At this moment we shall consider a minimizer sequence $(z_n)\subset \N^{-}$, i.e,
		$\ds\lim_{n\to\infty}J(z_n)=J^{-}$.
		Since $J$ is coercive in $\N$ and so on $\N^{-}$, using Lemma \ref{c1}, we can show that $(z_n)$ is a bounded sequence in $W$.
		Up to a subsequence we assume that $z_n\rightharpoonup z$ in $W$. Using the same ideas discussed in the proof of Theorem \ref{teorema1} we obtain
		\begin{equation}\label{lim1}
		\ds\lim_{n\to\infty}\Int Q(z_n)=\Int Q(z)\ \ \  \mbox{and} \ \ \ \ds\lim_{n\to\infty}\Int B(z_n)=\Int B(z).
		\end{equation}
		Furthermore, using \eqref{eq2} and $(\phi_3)$, it follows also that
		$$\begin{array}{rcl}
		J(z_n)&=&\Int A(z_n)-\frac{1}{q} B(z_n)+\left(\frac{1}{q}-\frac{1}{\alpha+\beta}\right) Q(z_n)\\[4ex]
		&\leq&
		\left(1-\Fr{\min\{\ell_i\}}{q}\right)\Int A(z_n)+\left(\Fr{1}{q}-\Fr{1}{\alpha+\beta}\right)\Int Q(z_n).
		\end{array}$$
		From the last estimate, using the fact $\left(1-\Fr{\min\{\ell_i\}}{q}\right)<0$, $J^{-}>0$ and \eqref{lim1} we deduce that $\Int Q(z)> 0$.  As a product the fibering map $\gamma_{z}$ admits an unique critical point $t_{1} > 0$ in such way that $t_1 z\in\N^{-}$.
		
		At this stage, arguing by contradiction, we assume that $z_n\not\rightarrow z$ in $W$. Since $(z_n)\subset \N^{-}$ we also mention that $$J(z_n)\geq J(sz_n),\ \ \forall s>0.$$ Therefore, using the last inequality and \eqref{liminf1}, we have been proved that
		$$\begin{array}{rcl}
		J(t_1z)&<&\ds\liminf_{n\rightarrow \infty}\Int A(t_1z_n)-\lim_{n\rightarrow \infty} \left[\Int \Fr{1}{q}K(t_1z_n)+\Fr{1}{\alpha+\beta}Q(t_1z_n)\right]\\[2ex]
		&\leq&\ds\liminf_{n\rightarrow \infty}\left[\Int A(t_1z_n) - \Fr{1}{q}K(t_1z_n)+\Fr{1}{\alpha+\beta}Q(t_1z_n)\right]\\[3ex]
		&=&\ds\liminf_{n\rightarrow \infty} J(t_1z_n)\leq \lim_{n\rightarrow \infty} J(z_n)=J^{-}.
		\end{array}$$
		This is a contradiction due the fact that $(z_n)$ is minimizer sequence. To sum up, we have been showed that $z_n \to z$ in $W$.
	\end{proof}
	
	\subsection{The second weak solution for the nonsingular case:}
	Here we stress out that $q>1$. Using
	standard arguments based on the Nehari method we consider one more time a minimization argument.
	Arguing as was made above, taking a subsequence if necessary, there exists $z\in W$ such that
	$z_n\rightharpoonup z \,\, \mbox{in } \,\, W. $
	It follows from the Proposition \ref{strong-converg-} that $z_n\rightarrow z ~\mbox{in}~W.$
	Furthermore, the last assertion says also that
	$$J(z)=\lim_{n\rightarrow \infty} J(z_n)=\ds\inf_{z\in\N^-} J(z)=J^-.$$
	Hence, applying Lemma \ref{criticalpoint}, we have that $z$ is a weak solution to the quasilinear elliptic System \eqref{eq1}. Since $J(z)=J(|z|)$ and $|z|=(|u|,|v|) \in \N^-$, we assume that $z$ is a nonnegative solution to the elliptic System \eqref{eq1}. It follows also from Lemma
	\ref{no-semitriv-} that $u,v\neq0$. This finishes the proof.

	\subsection{The second weak solution for the singular case:}
	One more time we consider the case singular case given by $0 < q < 1$. The main difficulty arises from the fact that
	the energy function $J$ is not in $C^{1}$ class. At this stage we assume that hypothesis $(C)$ holds true. The proof here is based on the same arguments employed in the proof of Theorem \ref{teorema2}. This ends the proof. \hfill\cqd

	\subsection{The proof of Theorem \ref{teorema1} and Theorem \ref{teorema2} completed:}
	In view of the previous sections there exist $\overline z\in \N^{+}$ and $\tilde z \in \N^{-}$ in such way that $$J(\overline{z})=\ds\inf_{w\in\N^{+}}J(w)\ \ \ \mbox{and}\ \ \  J(\tilde z)=\ds\inf_{ w\in\N^{-}}J(w).$$
	Here we mention that $\overline{z}$ and $\tilde{z}$ are critical points for $J$ which remains true for the nonsingular case or singular case. Furthermore, using the fact that $0 < \lambda+\mu < \lambda_{*} := \min(\eta_1, \eta_{2})$, we stress out that $\N^{+}\cap \N^{-}=\emptyset$. Therefore, $\overline z$ is a nonnegative ground state solution for the quasilinear elliptic System \eqref{eq1}.  As was mentioned before, using the fact that
	$$ J(w)=J(|w|)\,\, \mbox{and}\,\, J^{\prime}(w)=J^{\prime}(|w|)$$
	holds true for any $w \in \w$ we can assume $\overline z, \tilde z\geq 0$ in $\Omega$.
	It is worthwhile to mention that under hypothesis $(C)$ whenever $0< q < 1$ holds, we can use the same ideas discussed in previous sections we find at least two nonnegative weak solutions to the elliptic System \eqref{eq1}. This ends the proof.
	\hfill\cqd
	\section{Appendix: Orlicz-Sobolev spaces}
	
	The reader is  referred to  \cite{A,Rao1} regarding Orlicz and Orlicz-Sobolev spaces.  The usual norm on $L_{\Phi}(\Omega)$ is the Luxemburg norm given by
	$$
	\|u\|_\Phi=\inf\left\{\lambda>0~|~\int_\Omega \Phi\left(\frac{u(x)}{\lambda}\right)  \leq 1\right\}.
	$$
	The Orlicz-Sobolev norm of $ W^{1, \Phi}(\Omega)$ is defined by
	\[
	\displaystyle \|u\|_{1,\Phi}=\|u\|_\Phi+\sum_{i=1}^N\left\|\frac{\partial u}{\partial x_i}\right\|_\Phi.
	\]
	Recall that
	$$
	\widetilde{\Phi}(t) = \displaystyle \max_{s \geq 0} \{ts - \Phi(s) \},~ t \geq 0.
	$$
	It turns out that  $\Phi$ and $\widetilde{\Phi}$  are  N-functions  satisfying  the $\Delta_2$-condition, see \cite{Rao1}.
	Furthermore,  we mention that  $L_{\Phi}(\Omega)$  and $W^{1,\Phi}(\Omega)$  are separable, reflexive,  Banach spaces.

	Using the Poincar\'e inequality for the $\Phi$-Laplacian operator  it follows that
	\[
	\|u\|_\Phi\leq C \|\nabla u\|_\Phi~\mbox{for any}~ u \in W_{0}^{1,\Phi}(\Omega)
	\]
	holds true for some $C > 0$, see Gossez \cite{Gz1,gossez-Czech}.
	As a consequence,  $\|u\| :=\|\nabla u\|_\Phi$ defines a norm in $W_{0}^{1,\Phi}(\Omega)$ which is equivalent to $\|.\|_{1,\Phi}$. Let $\Phi_*$ be the inverse of the function
	$$
	t\in(0,\infty)\mapsto\int_0^t\frac{\Phi^{-1}(s)}{s^{\frac{N+1}{N}}}ds
	$$
	which extends to ${\mathbb{R}}$ by  $\Phi_*(t)=\Phi_*(-t)$ for  $t\leq 0$. We say that a $N$-function $\Psi$ grow essentially more slowly than $\Phi_*$, we write $\Psi<<\Phi_*$ if
	$$
	\lim_{t\rightarrow \infty}\frac{\Psi(\lambda t)}{\Phi_*(t)}=0,~~\mbox{for all}~~\lambda >0.
	$$
	
	The embedding below (cf. \cite{A, DT}) is used in the present work.
	$$
	\displaystyle W_{0}^{1,\Phi}(\Omega) \stackrel{\tiny cpt}\hookrightarrow L_\Psi(\Omega),~~\mbox{if}~~\Psi<<\Phi_*.
	$$
	In particular, as $\Phi<<\Phi_*$ (cf. \cite[Lemma 4.14]{Gz1}),
	$$
	W_{0}^{1,\Phi}(\Omega) \stackrel{\tiny{cpt}} \hookrightarrow L_\Phi(\Omega).
	$$
	Furthermore, we also mention that
	$$
	W_0^{1,\Phi}(\Omega) \stackrel{\mbox{\tiny cont}}{\hookrightarrow} L_{\Phi_*}(\Omega).
	$$
	
	For the next result we consider some estimate relate to $\Phi$ and $\Phi_{*}$ which are useful in the present work.
	
	\begin{rmk}\label{rmk-psi}
		The function $\psi(t) = t^{r-1}, r \in [1, \ell^*)$ satisfies $\Psi<<\Phi_*$ where $\Psi(t) = \int_{0}^{t} \psi(s) ds, t \in \mathbb{R}$. In other words,  the function $\Psi$ grow essentially more slowly than $\Phi_*$. In fact, we easily see that
		$$\lim_{t\rightarrow\infty}\frac{\Psi(\lambda t)}{\Phi_*(t)}\leq \frac{\lambda^{r}}{r\Phi_*(1)}\lim_{t\rightarrow\infty}\frac{1}{t^{\ell^*-r}}=0,~~\mbox{for all}~~\lambda>0.$$
		In that case $W_{0}^{1,\Phi}(\Omega) \stackrel{cpt}\hookrightarrow L_\Psi(\Omega)$.
	\end{rmk}
	
	\begin{rmk}\label{conseqphi3}
		Under assumption $(\phi_{3})$ we observe that
		\begin{equation*}
		\ell-2\leq\Fr{\phi^\prime(t)t}{\phi(t)}\leq m-2,\ \
		\ell\leq\Fr{\phi(t)t^2}{\Phi(t)}\leq m, t > 0.
		\end{equation*}
		Moreover, we have that
		$$\left\{\begin{array}{rcl}
		t^2\phi^{\prime\prime}(t)&\leq& (m-4)t\phi^\prime(t)+(m-2)\phi(t)\\
		t^2\phi^{\prime\prime}(t)&\geq& (\ell-4)t\phi^\prime(t)+(\ell-2)\phi(t), t \geq 0.
		\end{array}\right.$$
	\end{rmk}
	
	Now we refer the reader to \cite{Fuk_1,fang} for the some results elementary results on Orlicz and Orlicz-Sobolev spaces.
	\vskip.2cm
	
	\begin{prop}\label{lema_naru}
		Assume that  $\phi$ satisfies  $(\phi_1)-(\phi_3)$.
		Set
		$$
		\zeta_0(t)=\min\{t^\ell,t^m\},~~~ \zeta_1(t)=\max\{t^\ell,t^m\},~~ t\geq 0.
		$$
		\nd Then  $\Phi$ satisfies
		$$
		\zeta_0(t)\Phi(\rho)\leq\Phi(\rho t)\leq \zeta_1(t)\Phi(\rho),~~ \rho, t> 0,
		$$
		$$
		\zeta_0(\|u\|_{\Phi})\leq\int_\Omega\Phi(u)\leq \zeta_1(\|u\|_{\Phi}),~ u\in L_{\Phi}(\Omega).
		$$
	\end{prop}

	\begin{prop}\label{fang}
		Assume that $(\phi_1)-(\phi_3)$ holds.
		Define the function
		$$
		\eta_0(t)=\min\{t^{\ell-2},t^{m -2}\},~~~ \eta_1(t)=\max\{t^{\ell -2},t^{m - 2}\},~~ t\geq 0.
		$$
		\nd Then the function $\phi$ verifies
		$$
		\eta_0(t)\phi(\rho) \leq \phi(\rho t)\leq \eta_1(t)\phi(\rho),~~ \rho, t> 0,
		$$
	\end{prop}
	
	Using the last results we obtain some estimates in order to apply in Orlicz and Orlicz-Sobolev framework for the quasilinear elliptic Systems \eqref{eq1}.
	\begin{prop}\label{estimativanorma}
		Assume that $(\phi_1)-(\phi_3)$ Holds. Then there exist positives constants  $A_i(\ell_i,m_i):=A_i$ in such way that		$$\begin{array}{rcl}A_1\min\left\{||(u,v)||^{\displaystyle\min\{\ell_i\}},||(u,v)||^{\displaystyle\max\{m_i\}}\right\}&\leq& \Int \Phi_1(|\nabla u|)+\Phi_2(|\nabla v|)\\
		&\leq& A_2\max\left\{||(u,v)||^{\displaystyle\min\{\ell_i\}},||(u,v)||^{\displaystyle\max\{m_i\}}\right\},  (u, v) \in W.
		\end{array}$$
	\end{prop}
	In addition, the critical function $\Phi_{*}$ we need to consider some other estimates as follows
	\begin{prop}\label{lema_naru_*}
		Assume that  $\phi$ satisfies $(\phi_1)-(\phi_3)$.  Set
		$$
		\zeta_2(t)=\min\{t^{\ell^*},t^{m^*}\},~~ \zeta_3(t)=\max\{t^{\ell^*},t^{m^*}\},~~  t\geq 0
		$$
		\nd where $1<\ell,m<N$ and $m^* = \frac{mN}{N-m}$, $\ell^* = \frac{\ell N}{N-\ell}$.  Then
		$$
		\ell^*\leq\frac{t^2\Phi^\prime_*(t)}{\Phi_*(t)}\leq m^*;
		~~~~~~~~~
		\zeta_2(t)\Phi_*(\rho)\leq\Phi_*(\rho t)\leq \zeta_3(t)\Phi_*(\rho),~~ \rho, t> 0,
		$$
		and
		$$
		\zeta_2(\|u\|_{\Phi_{*}})\leq\int_\Omega\Phi_{*}(u)\leq \zeta_3(\|u\|_{\Phi_*}),~ u\in L_{\Phi_*}(\Omega).
		$$
	\end{prop}
	\textbf{Acknowledgments}: The authors were partially supported by Fapeg/CNpq grants 03/2015-PPP.

\end{document}